\documentclass[11pt]{amsart}
\usepackage{amsfonts}
\usepackage{tikz}
\usepackage{xcolor}
\usepackage{ifthen}
\usepackage{amsmath}
\usepackage{amsthm}
\usepackage{amssymb}
\usepackage{tikz-cd} 
\usepackage{bbm}
\theoremstyle{plain}
\newtheorem{theorem}{\bf{Theorem}}[section]
\newtheorem{lemma}[theorem]{Lemma}
\newtheorem{proposition}[theorem]{Proposition}
\newtheorem{corollary}[theorem]{Corollary}
\newtheorem{remark}[theorem]{Remark}
\newtheorem{fact}[theorem]{Fact}
\newtheorem{definition}[theorem]{Definition}

\usepackage[top=.5 in,bottom=.5 in, left=.5 in, right= .5 in]{geometry}

\def\R{{\mathbb R}}
\def\C{{\mathbb C}}

\newcommand{\mrz}{M_{\bR}(z)}
\newcommand{\bX}{\mathbf{X}}
\newcommand{\bY}{\mathbf{Y}}
\newcommand{\bZ}{\mathbf{Z}}
\newcommand{\bV}{\mathbf{V}}
\newcommand{\bU}{\mathbf{U}}
\newcommand{\bI}{\mathbf{I}}

\newcommand{\bR}{\mathbf{R}}

\newcommand{\gx}{G_{\bX}(z)}

\newcommand{\muz}{M_{\bU}(z)}
\newcommand{\musz}{M^2_{\bU}(z)}
\newcommand{\muw}{M_{\bU}(w)}

\newcommand{\oj}{\omega_1(z)}
\newcommand{\ojw}{\omega_1(w)}
\newcommand{\od}{\omega_2(z)}
\newcommand{\odw}{\omega_2(w)}
\newcommand{\cgw}{\textrm{C}^\ast}
\newcommand{\calA}{\mathcal{A}}

\numberwithin{equation}{section}

\begin{document}
	\title{The Kummer distribution in free probability, and its characterizations}

	\author[M.\'{S}wieca]{Marcin \'{S}wieca}
	\address{Wydzia{\l} Matematyki i Nauk Informacyjnych\\
		Politechnika Warszawska\\
		ul. Ko\-szy\-ko\-wa 75\\
		00-662 Warsaw, Poland} \email{marcin.swieca@pw.edu.pl}
	\thanks{This research was funded in part by National Science Centre, Poland WEAVE-UNISONO grant BOOMER 2022/04/Y/ST1/00008.}
	\thanks{For the purpose of Open Access, the authors have applied a CC-BY public copyright licence to any Author Accepted Manuscript (AAM) version arising from this submission.}
	
	\date{\today}
	\begin{abstract}
		We study the analogue of Kummer distribution in  free probability. We prove characterization of free-Kummer and free Poisson distributions by freeness properties together with some assumptions about conditional moments. Our main tools are subordination of free multiplicative convolution and Boolean cumulants.
	\end{abstract}
	
	\maketitle
	\section{Introduction}
	Recall that the Gamma law $G(a,c)$  with parameters $a,c>0$ is a probability measure that has the density
	$$\frac{c^a}{\Gamma(a)}x^{a-1}e^{-cx}\mathbbm{1}_{(0,\infty)}(x)$$
	and the Kummer law $K(a,b,c)$ with parameters $a,c>0$, $b\in\mathbb{R}$  is a probability  measure that has the density $$C\frac{x^{a-1}}{(1+x)^{a+b}}e^{-cx}\mathbbm{1}_{(0,\infty)}(x),$$
	where $C=C(a,b,c)$ is a normalizing constant.
	
	It was observed in \cite{HamzaValois} by Hamza and Vallois  that if $X, Y$ are independent, $X\sim K(a,b,c)$ and $\ Y\sim G(a+b,c)$ then
	\begin{equation*}
		U=\frac{Y}{1+X}\ \ \ \mbox{and}\ \ \ V=X(1+U)
	\end{equation*}
	are independent and distributed as $U\sim K(a+b,-b,c)$ and $V\sim G(a,c)$. We will refer to this as HV property. In \cite{PiliWeso1} the converse was proved under some technical assumptions for densities of $X$ and $Y$, namely assume that $X$ and $Y$ are independent, positive and such that logarithms of their densities are locally integrable on $\R^+$.  If $U$ and $V$ are independent  then  $X\sim K(a,b-a,c),\ Y\sim G(b,c)$ for some parameters $a,b,c>0$. Later the same authors dropped  the assumption of  existence of densities and replaced  independence of $U$ and $V$ by constancy of regressions. See  \cite{PiliWeso2} for more details.
	
	The  HV property has its analogue for random matrices. It involves matrix version of Kummer distribution and Wishart distribution. Before we state the property  we recall some basic facts. Let $\Omega$ denote the open cone of real, symmetric and positive definite matrices. Then  a matrix Kummer distribution  $K_n(a,b,\Sigma)$ with parameters $a>\frac{n-1}{2}, b\in\R,\Sigma\in\Omega$ is a probability measure supported on $\overline{\Omega}$ that has the  density
	$$c (\det x)^{a-\frac{n+1}{2}}(\det(I+x))^{-a-b}e^{-\left<\Sigma,x\right>}\mathbbm{1}_{\Omega}(x),$$
	where $I$ is the $n\times n$ identity matrix, $\left<x,y\right>=\textrm{tr}(xy)$ and $c=c(a,b,\Sigma)$ is  a normalizing constant. Similarly a Wishart distribution $W_n(b,\Sigma)$ with parameters  $b\in\left\{0,\frac{1}{2},1,\frac{3}{2},\ldots,\frac{n-1}{2}\right\}\cup\left(\frac{n-1}{2},\infty\right)$ and $\Sigma\in\Omega$ is  a probability measure supported on $\overline{\Omega}$ with Laplace transform given by the following formula
	$$\mathbb{E}e^{\left<s,Y\right>}=\left(\frac{\det \Sigma}{\det(\Sigma-s)}\right)^b,\ \mbox{for all}\ s\in\Omega\ \mbox{such that }\ \Sigma-s\in\Omega.$$
	When $b>\frac{n-1}{2}$ the Wishart distribution  $W_n(b,\Sigma)$ has the density
	$$\frac{(\det \Sigma)^b}{\Gamma_{\Omega}(b)} (\det x)^{b-\frac{n+1}{2}}e^{-\left<\Sigma,x\right>}\mathbbm{1}_{\Omega}(x),$$
	where $\Gamma_{\Omega}(b)=\pi^{n(n-1)/4}\prod_{k=1}^n\Gamma\left(b-\frac{k-1}{2}\right)$ is the multivariate Gamma function. 
	
	The HV property for random matrices was established in \cite{PiliKolo} and it states  that if $X\sim K_n(a,b,c I)$ and  $Y\sim W_n(a+b,c I)$ with  $a>\frac{n-1}{2}$, $a+b>\frac{n-1}{2}$, $c>0$ are independent then $U$ and $V$ defined as
	$$U=(I+X)^{-\frac{1}{2}} Y (I+X)^{-\frac{1}{2}}, \ \ V=(I+U)^{\frac{1}{2}}X(I+U)^{\frac{1}{2}},$$
	are also independent and distributed as $K_n(a+b,-b,c I)$ and $W_n(a,c I)$ respectively. In the same paper the converse (characterization) was established under the assumption that $\bX,\bY$ have positive densities.
	
	The HV property in free probability was proved using the concept of asymptotic freeness that, roughly speaking, states that large independent random matrices behave like free random variables. To be more precise, for a $n\times n$ matrix $A$ let  $\mu_A$ be  an empirical spectral distribution of $A$ i.e.
	$$\mu_A=\frac{1}{n}\sum_{k=1}^n\delta_{\lambda_k},$$
	where $\lambda_1,\lambda_2,\ldots,\lambda_n$ are all  eigenvalues of $A$. Then the following well known theorem proved by Voiculescu \cite{VoicAsymptoticFreeness} holds.
	\begin{theorem}{}
		Suppose for all $n\geq1 $ the random matrices  $X_n$ and $Y_n$ are independent and the distribution of  $Y_n$ is unitary invariant. Assume also that the sequences $(X_n)_n$ and $(Y_n)$ have almost surely an asymptotic spectral distributions, say $\mu_{X_n}\stackrel{1}{\to}\mu$ and  $\mu_{Y_n}\stackrel{1}{\to}\nu$ weakly. Then for all polynomials $P(x,y)$ of two non-commutative variables almost surely we have $$\lim\limits_{n\to\infty}\frac{1}{n}\mbox{tr}\left(P(X_n,Y_n)\right)=\varphi(P(\bX,\bY)),$$
		where $\bX,\bY$ are two free random variables in a non-commutative probability space $(\mathcal{A},\varphi)$ such that $\bX\sim\mu$ and $\bY\sim\nu$.
	\end{theorem}
	
	The existence of a limiting spectral distribution for Kummer matrices was established in \cite{PiliszekFreeKummer}. It was  proved that for  $X_n\sim K_n(a_n,b_n,c_n)$ where parameters $a_n>\frac{n-1}{2}$, $b_n\in\R$, $c_n>0$  are such that $\frac{2a_n}{n}\to \alpha>1$, $\frac{2b_n}{n}\to \beta$, $\frac{2 a_n}{n}\to \gamma>0$ then the sequence of empirical spectral distributions $\mu_{X_n}$ converges almost surely to what we call in this paper  the free-Kummer distribution $\mathcal{K}(\alpha,\alpha +\beta,\gamma)$. Recall also the  well known fact that  for Wishart matrices $Y_n\sim W_{n}(\lambda_n,\alpha_n I)$ with  $\frac{2\lambda_n}{n}\to\lambda$ and $\frac{2\alpha_n}{n}\to \alpha$ the limiting empirical spectral distribution is the free Poisson (Marchenko-Pastur) distribution  $\nu(\lambda,1/\alpha)$.
	
	In the same paper the HV property in free probability was proved.
	
	\begin{theorem}[\cite{PiliszekFreeKummer}]\label{FreeHVproperty}
		Let $\bX$ and $\bY$ be self-adjoint random variables in some $C^*$ probability space, such that $\bX$ has the free-Kummer distribution $\mathcal{K}(\alpha,\alpha +\beta,\gamma)$ and the distribution of $\bY$ is free-Poisson $\nu(\alpha+\beta,1/\gamma,)$ for some $\alpha>1$, $\gamma>0$ and $\beta>1-\alpha$. If $\bX, \bY$ are free, then $\bU$ and $\bV$ defined as
		\begin{equation}
			\bU=(\bI+\bX)^{-\frac{1}{2}}\bY(\bI+\bX)^{-\frac{1}{2}}\ \mbox{and}\ \bV=(\bI+\bU)^{\frac{1}{2}}\bX(\bI+\bU)^{\frac{1}{2}}\label{defofUV}
		\end{equation}
		are free. Moreover $\bU\sim \mathcal{K}(\alpha+\beta,\alpha,\gamma)$ and $\bV\sim \nu(\alpha,1/\gamma)$.
	\end{theorem}

 In this paper we show that  the definition of free-Kummer distribution  $\mathcal{K}(\alpha,\beta,\gamma)$ can be extended in a consistent way also for  $\alpha\in(0,1]$, $\beta\in \R$ and $\gamma>0$ and
  prove the following characterization theorem which can be thought of as the  converse statement to Theorem \ref{FreeHVproperty}. Note that freeness of  of $\bU$ and $\bV$ is replaced by weaker condition of constancy of (two) conditional moments of $\bV$ given $\bU$. 
 	\begin{theorem}\label{Thgeneral}
 	Let $(k,l)\in\left\{(1,-1),(1,2),(-1,-2)\right\}$ be fixed and let $\bX,\bY$ be free, positive, non-degenerated and self-adjoint random variables  in some $W^*$ probability space $(\mathcal{A},\varphi)$.  If $\bU,\bV$ are defined by \eqref{defofUV} and 
 	\begin{equation*}	\varphi\left(\bV^k\mid \bU\right)=m_k\bI,\ \ \mbox{and}\ \   \varphi\left(\bV^l\mid \bU\right)=m_l\bI, \end{equation*}
 	for some constants $m_k,m_l\in\R$ then $\bX$ has free-Kummer distribution   $\mathcal{K}\left(\alpha,\beta,\gamma\right)$ and $\bY$ has  free Poisson distribution $\nu\left(\beta,1/\gamma\right)$  for some $\alpha,\beta,\gamma>0$.
 \end{theorem}
	It is worth to note that in Theorem \ref{FreeHVproperty} both $\alpha$ and $\alpha+\beta$ are assumed to be greater than $1$ and consequently both   $\bX$ and $\bY$ are invertible. Our Theorem \ref{Thgeneral} suggest that  this theorem holds more generally for $\alpha,\alpha+\beta>0$. See Remark \ref{Uwagaoalfie} for more details.
	
	Our main technical result needed for the proof of Theorem \ref{Thgeneral} is calculation  of the following expression 
	\begin{equation}
	\varphi\left(g_1(\bR)\left(\bI-z\bR^{\frac{1}{2}}\bY\bR^{\frac{1}{2}}\right)^{-1}g_2(\bR)\left(\bI-w\bR^{\frac{1}{2}}\bY\bR^{\frac{1}{2}}\right)^{-1}\right)\  \textrm{for}\ \ z,w\in\C\setminus \R_+,\label{wzorek1}
	\end{equation}
	in terms of the subordination functions for free multiplicative convolution. Here $\bR,\bY$ are assumed to be free and positive non-commutative random variables  and $g_1(\bR),g_2(\bR)$ belong to unital subalgebra generated by $\bR$. This is done in Theorem \ref{ThMainExpectation}.
	
	We also prove  that (for $\beta>0$) the distribution $\mathcal{K}(\alpha,\beta,\gamma)$ is the only distribution supported on $[0,\infty)$ for which the Cauchy-Stieltjes transform $G(z)$   the  equation that has the form 
		\begin{equation}
		z(z+1)G^2(z)-(\gamma z(z+1)-(\alpha-1) (z+1)+\beta z)G(z)+\gamma z+\delta=0,
	\end{equation}
	for some $\delta\in\R$. See Proposition \ref{kummerlemma} for more details. Problems of this type where the Cauchy-Stieltjes transform emerges as solution to a certain quadratic equation  that depends on unknown parameters, appear frequently in characterization problems in free probability. See for example already mentioned article \cite{PiliszekFreeKummer} or \cite{szpojankowski2017matsumoto}  where characterization the free-GiG and free-Poisson distributions was studied. The proof of uniqueness in those two papers rely on the   study of roots of $\Delta(z)$-the discriminant of equation for $G(z)$ and the geometric argument is given to show that $\Delta(z)$ has a desired form. In our paper we present new and a very simple analytic proof of this problem for $\mathcal{K}(\alpha,\beta,\gamma)$ that avoids geometric arguments.
	
	The organization of this paper is as follows. In chapter 2 we recall some basic facts from free probability theory that are needed to understand this article. In chapter 3 we discuss definitions and properties of the free-Kummer definition and the free-Poisson distributions. In Chapter 4 we compute  the  expression given by formula \ref{wzorek1}. Chapter 5 is devoted to characterization theorems and the proof of Theorem \ref{Thgeneral}.
	\section{Background and notation}\label{background}
	In this section we introduce basic notions and facts from non-commutative probability theory that are needed to understand this paper. We assume we are given a $\cgw$-probability space $\left(\calA,\varphi\right)$ i.e. $\calA$ is a unital $\cgw$-algebra and $\varphi:\calA\to\mathbb{C}$ is positive, tracial and faithful linear functional (state) such that $\varphi(\bI)=1$ where $\bI$ is  the unit of $\calA$.
	
	Elements of $\calA$ are called (non-commutative) random variables and in this paper are denoted as $\bX, \bY,\bZ$ etc.
	\subsection{Freeness and Boolean cumulants}
	Freeness is one of the basic concepts that serves as the analogue  of independence from classic probability theory and was introduced by Voiculescu in \cite{voiculescu1986addition}.
	\begin{definition}
		We say that unital  subalgebras $\calA_1,...,\calA_n$ of $\calA$ are free if for every $m$ and for every  choice of centered random variables $\bX_k\in\calA_{i_k}$  \textrm{i.e.}  $\varphi(\bX_k)=0$, $k=1,2,\ldots,m$, such that $i_1\neq i_2\neq\ldots\neq i_m$ we have 
		$$\varphi\left(\bX_1\bX_2\cdots\bX_m\right)=0.$$
		
		We say that random variables $\bX, \bY\in\calA $ are free if unital subalgebras generated by those elements are free. 
	\end{definition}
	The definition of freeness can be viewed as a rule for computing joint moments. For example if $\bX, \bY$ are free, then $\varphi(\bX \bY)=\varphi(\bX)\varphi(\bY)$.
	
	For a positive integer $n$ let us denote $[n]=\left\{1,2,\ldots,n\right\}$.  
	\begin{definition}
		\hfill
		\begin{enumerate}
			\item A partition $\pi$ of $[n]$ is a set $\pi=\left\{B_1,...,B_k\right\}$ of non-empty and pairwise disjoint subsets of $[n]$ such that $[n]=\bigcup_{i=1}^kB_i$. Elements $B_1,\ldots,B_k$ are called blocks of $\pi$. The set of all partition of $[n]$ is denoted by $\mathcal{P}(n)$.
			\item A partition $\pi\in \mathcal{P}(n)$ is called  an interval partition  if every block $B$ of $\pi$ is of the form $[n]\cap I$ for some interval $I$. The set of all $2^{n-1}$ interval partitions of $[n]$ is denoted by $Int(n)$.
		\end{enumerate}
	\end{definition}
	\begin{remark}
		The set $Int(n)$ has  a lattice structure induced by the so-called reversed refinement order. We say that $\pi_1\leq \pi_2$ if every block of the partition $\pi_1$ is contained in some block of $\pi_2$. 
		The partition $1_n$ with one block $[n]$ is the maximum element in this lattice.
	\end{remark}
	
	\begin{definition}\label{defboolcumu}
		For $n\geq 1$  the Boolean cumulant functional $\beta_n:\mathcal{A}^n\to\mathbb{C}$ is defined recursively by
		$$\forall \bX_1,\ldots,\bX_n\in\mathcal{A}:\varphi\left(\bX_1\bX_2\cdots\bX_n\right)=\sum_{\pi\in Int(n)}\beta_\pi\left(\bX_1,\ldots,\bX_n\right),$$
		
		where for $\pi=\left\{B_1,\ldots,B_k\right\}$ 
		$$\beta_\pi\left(\bX_1,\ldots,\bX_n\right)=\prod_{j=1}^{k}\beta_{|B_j|}\left(\bX_i:i\in B_j\right).$$
	\end{definition}
	In particular $\beta_1=\varphi$ and $\beta_2(\bX,\bY)=\varphi(\bX\bY)-\varphi(\bX)\varphi(\bY)$.
	
	We will need two formulas involving Boolean cumulants. They can be found in \cite{fevrier2020using} and \cite{lehner2019boolean} and were used also in \cite{szpojankowski2019conditional}.
	\begin{proposition}\label{propbool}
		Assume we are given two collections of random variables $\{\bX_1,\bX_2,\ldots,\bX_{n+1}\}$ and $\{\bY_1,\bY_2,\ldots,\bY_{n}\}$ that are free, $n\geq 1$. Then
		\begin{multline}\label{boolmain1}
			\varphi\left(\bY_1\bX_1\cdots\bY_n\bX_n\right)=\\
			=\sum_{k=0}^{n-1}\sum_{0=j_0<j_1<\ldots<j_{k+1}=n}\varphi\left(\bX_{j_1}\cdots\bX_{j_{k+1}}\right)\prod_{l=0}^k\beta_{2(j_{l+1}-j_l)-1}\left(\bY_{j_l+1},\bX_{j_l+1},\ldots, \bX_{j_{l+1}-1},\bY_{j_{l+1}}\right)
		\end{multline}
		and
		\begin{multline}\label{boolmain3}
			\beta_{2n+1}\left(\bX_1,\bY_1,\ldots,\bX_n,\bY_n,\bX_{n+1}\right)=\\
			=\sum_{k=2}^{n+1}\sum_{1=j_1<\ldots<j_{k}=n}\beta_k\left(\bX_{j_1},\ldots,\bX_{j_{k}}\right)\prod_{l=1}^{k-1}\beta_{2(j_{l+1}-j_l)-1}\left(\bY_{j_{l}},\bX_{j_l+1},\bY_{j_l+1},\ldots,\bX_{j_{l+1}-1}, \bY_{j_{l+1}-1}\right).
		\end{multline}
	\end{proposition}
	We will also need the following formula for Boolean cumulant with product as entries.
	
	\begin{proposition}[Proposition 2.12 in \cite{fevrier2020using}]\label{betanaproduktach}
		Let $1\leq i_1<i_2<...<i_m=n$ be positive integers.  Then
		\begin{equation*}
			\beta_m(\bX_1\cdots \bX_{i_1},\bX_{i_1+1}\cdots \bX_{i_2},...,\bX_{i_{m-1}+1}\cdots \bX_n)=\sum_{\stackrel{\pi\in Int(n)}{\sigma\vee\pi =1_n}}\beta_{\pi}(\bX_1,\bX_2,...,\bX_n)
		\end{equation*}
		where $$\sigma=\left\{\{1,2...,i_1\},\{i_1+1,...,i_2\},...,\{i_{m-1}+1,...,n\}\right\}\in Int(n)$$ and $\vee$ is join of partitions (smallest upper bound of two partitions in $Int(n)$).
	\end{proposition}
	
	\subsection{Conditional expectations}
	Assume that $(\calA,\varphi)$  is a $W^{*}$ probability space, i.e., $\calA$  is a finite von Neumann algebra and
	$\varphi$ a faithful, normal, tracial state. If  $\mathcal{B}\subset\calA$ is von Neumann subalgebra, we denote by $\varphi\left(\cdot\mid\mathcal{B}\right)$ the conditional expectation with respect to $\mathcal{B}$. That is $\varphi\left(\cdot\mid\mathcal{B}\right):\calA\to\mathcal{B}$ is a  faithful, normal projection such that $\varphi\circ\left[ \varphi\left(\cdot\mid\mathcal{B}\right)\right]=\varphi(\cdot)$. The map $ \varphi\left(\cdot\mid\mathcal{B}\right)$ is a $\mathcal{B}$-bimodule map i.e. $$\varphi\left(\bY_1\bX\bY_2\mid\mathcal{B}\right)=\bY_1\varphi\left(\bX\mid\mathcal{B}\right)\bY_2$$ for all $\bX\in\calA$ and $\bY_1,\bY_2\in\mathcal{B}$.  For the existence of the conditional expectation see e.g   (\cite{Takesaki1}, Proposition 2.36).
	
	The conditional expectation has the following important property
	\begin{equation}
		\varphi\left(\bX\mid\mathcal{B}\right)=\bY\iff\ \ \bY\in\mathcal{B}\ \ \mbox{and}\ \ \forall\, \bZ \in \mathcal{B}:\ \varphi(\bX\bZ)	=\varphi(\bY\bZ).\label{CondProp2}
	\end{equation}
	
	\subsection{Distribution of a random variable and analytic tools.}
	\begin{definition}
		The distribution of a  self-adjoint random variable $\bX\in\calA$ is a uniquely determined,  compactly supported, probability measure $\mu_\bX$ on the real line such that for all $n\geq 1$
		$$\varphi\left(\bX^n\right)=\int_{\mathbb{R}}x^n\mu_{\bX}(dx).$$
	\end{definition}
	We list here some analytic tools and their properties that we use in this paper.
	\begin{enumerate}
		\item The Cauchy-Stieltjes transform of  a compactly supported measure  $\mu$ on the real line is  the map 
		$$G_\mu(z)=\int_{\mathbb{R}}\frac{\mu(dx)}{z-x},$$
		defined for  $z\in\mathbb{C}\setminus\textrm{supp}(\mu)$.
		It is known that the Cauchy-Stieltjes transform is an analytic map on $\mathbb{C}\setminus\textrm{supp}(\mu)$ and $G_\mu:\mathbb{C}^+\to \mathbb{C}^-$.\newline
		If $\bX$ is a self-adjoint random variable we write $G_{\bX}$ for $G_{\mu_\bX}$. Note that
		$$\gx=\varphi\left((z\bI-\bX)^{-1}\right)=\int_{\mathbb{R}}\frac{\mu_\bX(dx)}{z-x}.$$
		
		\item The moment transform of $\bX$  is  defined  for all $z\in\mathbb{C}$ such that $\bI-z\bX$ is invertible  as  $$M_\bX(z)=\varphi\left(z\bX(\bI-z\bX)^{-1}\right).$$
		$M_\bX$ is an an analytic function in some  neighborhood of $0$ and one has 
		$$M_\bX(z)=	\sum_{k=1}^\infty\varphi(\bX^k)z^k.$$
		The moment transform $M_\bX$ and the Cauchy-Stieltjes $\gx$ transform are related by the equation
		\begin{equation}
			G_\bX\left(\frac{1}{z}\right)=z(1+M_\bX(z)).\label{M-to-G}
		\end{equation}
		\item For a self adjoint, positive but non-zero random variable  $\bX$ one can define so called $S$-transform. The $S$-transform  $S_{\bX}$ is defined  in the neighborhood of $(\mu_{\bX}(\{0\})-1,0)$ in $\C$ as
		\begin{equation*}
			S_\bX(z)=\frac{z+1}{z}M^{\langle -1\rangle}_\bX(z),
		\end{equation*}
		where $M^{\langle -1\rangle}_\bX$ is the inverse function of $M_\bX$. The $S$-transform has the following useful property
		\begin{equation}
			S_{\bX\bY}=S_\bX\cdot S_{\bY}\label{Stransformproperty}
		\end{equation} when $\bX$ and $\bY$ are free. 
		
		\item The $\eta$-transform of $\bX$ is defined by $$\eta_\bX(z)=\frac{M_\bX(z)}{M_\bX(z)+1}.$$
		In some neighborhood of $0$ one has 
		$$\eta_\bX(z)=\sum_{k=1}^\infty\beta_{k}(\bX,...,\bX)z^k.$$
	\end{enumerate}
	All of these transformations uniquely determine moments of a self-adjoint random variable $\bX$ and thus also uniquely determine its distribution.
	
	\subsection{Subordination}
	Let $\bX,\bY$ be free self-adjoint and positive random variables. In \cite{biane1998processes} Biane establish a fundamental connection between moment transforms of $\bX,\bY$ and $\bX\bY$. It involves two functions $\omega_1,\omega_2:\C\setminus\R_+\to\C$ which can be defined as unique holomorphic functions on  $\C\setminus\R_+$ satisfying $\omega_k{(\overline{z})}=\overline{\omega_k(z)}$ and $\arg(\omega_k(z))\geq \arg(z)$ for all $z\in \C^+$ and $k=1,2$ and
	\begin{equation}
		M_{\bX\bY}(z)=M_{\bY}(\oj)=M_{\bX}(\od).\label{subordiantion}
	\end{equation}
	Because of the last property $\omega_1,\omega_2$  are called the subordination functions.
	
	In the framework of von Neumann algebras \eqref{subordiantion} can be stated in a more general way
	\begin{equation}
		\varphi\left((\bI-z\bX^{\frac{1}{2}}\bY \bX^{\frac{1}{2}})^{-1}\mid\bX\right)=\left(\bI-\od\bX\right)^{-1}. \label{conditionalsubordination}
	\end{equation}
	
	The subordination functions $\oj, \od$ can be expanded near the origin  into a Taylor series with Bolean cumulants as coefficients:
	\begin{align}
		\omega_1(z)&=\sum_{n=1}^\infty\beta_{2n-1}(\bX,\bY,\bX,\ldots,\bY,\bX)z^{n},\label{rowzw1}\\
		\omega_2(z)&=\sum_{n=1}^\infty\beta_{2n-1}(\bY,\bX,\bY,\ldots,\bX,\bY)z^{n}.\label{rowzw2}
	\end{align}
	See \cite{lehner2019boolean} for details.
	
	The $S$-transform property $S_{\bX\bY}=S_\bX\cdot S_{\bY}$ and the subordination property \eqref{subordiantion} imply also the following useful identity
	\begin{equation}
		\oj\od=z\cdot\eta_{\bX\bY}(z)=z\frac{M_{\bX\bY}(z)}{M_{\bX\bY}(z)+1}, \ \mbox{for}\ z\in \C\setminus\R_+.\label{usefulidentity}
	\end{equation}
	
	\section{Free–Poisson and free–Kummer distributions and their properties}
	
	\subsection{Free-Poisson Distribution} We say that a probability measure $\nu$   is free Poisson $\nu(\lambda,\gamma)$  or Marchenko-Pastur distribution with parameters $\lambda\geq 0, \gamma>0$  if
	$$\nu=\max\{0,1-\lambda\}\delta_0+\lambda\nu_1,$$
	where $\nu_1$ is a measure with density
	$$\frac{1}{2\pi\gamma x}\sqrt{4\lambda\gamma^2-(x-\gamma(1+\lambda))^2}\ \mathbbm{1}_{\left(\gamma(1-\sqrt{\lambda})^2,\gamma(1+\sqrt{\lambda})^2\right)}(x).$$
	\newline
	The $S$-transform of the free Poisson distribution $\nu(\lambda,\gamma)$ is equal $$S(z)=\frac{1}{\gamma z+\lambda\gamma}.$$
	\subsection{Free-Kummer distribution}
	For  $\alpha>0$, $\alpha\neq 1$,  $\gamma>0$, $\beta\in\R$ the free-Kummer distribution $\mathcal{K}(\alpha,\beta,\gamma)$  is a probability measure $\mu_{\alpha,\beta,\gamma}$ defined as
	$$\mu_{\alpha,\beta,\gamma}=\max\{0,1-\alpha\}\delta_0+\alpha\mu_1,$$
	where  $\mu_1$ is a measure with the density
	\begin{equation}
		\frac{1}{2\pi\alpha}\sqrt{(x-a)(b-x)}\left(\tfrac{|\alpha-1|}{x\sqrt{ab}}-\tfrac{\beta}{(1+x)\sqrt{(a+1)(b+1)}}\right)\mathbbm{1}_{(a,b)}(x),\label{KummerDensity}
	\end{equation}
	and  $(a,b) $ is the unique solution of  
	\begin{equation}\label{freeKummerab}
		\left\{\begin{array}{lc}
			\gamma+\frac{\beta}{\sqrt{(a+1)(b+1)}}-\frac{|\alpha-1|}{\sqrt{ab}}&=0,\\
			\gamma\frac{a+b}{2}-\alpha+1+\beta-\frac{\beta}{\sqrt{(a+1)(b+1)}}&=2,
		\end{array}\right.
	\end{equation}
	satisfying $0<a<b$.
	The free-Kummer distribution  was defined for $\alpha>1$ in \cite{PiliszekFreeKummer} as a limit of empirical  spectral distribution of Kummer matrices. The above definition extends the definition from that paper for $\alpha\in(0,1)$ and one can easily check that for $\alpha\in (0,1)$, $\beta \in\R$, $\gamma>0$ we have
	\begin{equation*}
		\mu_{\alpha,\beta,\gamma}=(1-\alpha)\delta_0+\alpha\mu_{\frac{1}{\alpha},\frac{\beta}{\alpha},\frac{\gamma}{\alpha}}.  
	\end{equation*}
	
	We note here that in \cite{PiliszekFreeKummer} the parameter $\alpha$ was  assumed to be positive but not necessarily greater than $1$. One can  see that the condition $\alpha\leq 1$  contradicts some assumptions made in that paper and the definition of free-Kummer distribution does not work in this case.   To reflect this fact we changed formulation of Theorem \ref{FreeHVproperty} by adding assmuptions $\alpha,\alpha+\beta>1$. 
	
	We will show now how to extend the definition of  $\mathcal{K}(\alpha,\beta,\gamma)$ for $\alpha=1$. The reader can find more details in the Appendix. 
	
	For $\alpha=1$ and $\beta<1-(1+\sqrt{\gamma})^2$ the above definition still works and in this case $\mathcal{K}(1,\beta,\gamma)$ is the distribution of $X-1$ where $X$ has the free-Poisson distribution $X\sim  \nu(1-\beta,1/\gamma)$.
	
	For $\alpha=1$ and $\beta\geq1-(1+\sqrt{\gamma})^2$  we define   $\mathcal{K}(1,\beta,\gamma)$
	to be a probability measure which has the density
	\begin{equation}
		\frac{1}{2\pi}\sqrt{x(b-x)}\left(\tfrac{\sigma}{x}-\tfrac{\beta}{(1+x)\sqrt{b+1}}\right)\mathbbm{1}_{(0,b)}(x),\label{kummerdensity1}
	\end{equation}
	where  $\sigma=\gamma+\frac{\beta}{\sqrt{b+1}}$ and $b$ is the unique positive solution of \begin{equation}
		\gamma\frac{b}{2}+\beta-\frac{\beta}{\sqrt{(b+1)}}=2.\label{Eqforb}
	\end{equation} 
	
	The  Cauchy-Stieltjes transform of $\mathcal{K}(\alpha,\beta,\gamma)$ for $\alpha>1$ was calculated in \cite{PiliszekFreeKummer} and is  given by the following formula
	\begin{equation}
		G_{\alpha,
			\beta,\gamma}(z)=\tfrac{1}{2}\left[\gamma-\tfrac{\alpha-1}{z}+\tfrac{\beta}{z+1}+\sqrt{(z-a)(z-b)}\left(\tfrac{\beta}{(z+1)\sqrt{(a+1)(b+1)}}-\tfrac{|\alpha-1|}{z\sqrt{ab}}\right)\right].\label{CauchyStieltjesKummer}
	\end{equation}
	For $0<\alpha<1$ we have $G_{\alpha,\beta,\gamma}(z)=\frac{1-\alpha}{z}+\alpha G_{\frac{1}{\alpha},\frac{\beta}{\alpha},\frac{\gamma}{\alpha}}(z)$ and in consequence the formula \eqref{CauchyStieltjesKummer} remains valid also  for $\alpha\in(0,1)$. 
	
	For $\alpha=1$, $\beta>1-(1+\sqrt{\gamma})^2$, $\gamma>0$  one can check that 
	\begin{equation*}
		G_{\alpha,\beta,\gamma}(z)=\tfrac{1}{2}\left[\gamma+\tfrac{\beta}{z+1}+\sqrt{z(z-b)}\left(\tfrac{\beta}{(z+1)\sqrt{b+1}}-\tfrac{\sigma}{z}\right)\right]. \label{CauchyStieltjesKummer2}
	\end{equation*}
	\begin{remark}\label{EqForGremak}
		Using \eqref{freeKummerab} for $\alpha \neq 1$ or \eqref{Eqforb} when $\alpha=1$ one can check that $G=G_{\alpha,
			\beta,\gamma}$ solves the following quadratic equation
		\begin{equation}
			z(z+1)G^2(z)-(\gamma z(z+1)-(\alpha-1) (z+1)+\beta z)G(z)+\gamma z+\delta=0.\label{EquationforG}
		\end{equation}
		Later we will show that, under some conditions, this equation characterizes free-Kummer distribution. See Proposition \ref{kummerlemma} for more details.	
	\end{remark}

	\section{The main technical results}
	Let $\bR$ and $\bY$ be free,  self-adjoint and positive random variables from some $W^*$ probability space $(\mathcal{A},\varphi)$.  Let $g_1(\bR)$ and $g_2(\bR)$  be two elements in the unital subalgebra generated by $\bR$. We also denote  $\bU=\bR^{\frac{1}{2}}\bY\bR^{\frac{1}{2}}$. The goal of this section is to  evaluate the  function following function in terms of subordination functions
	\begin{equation}
		k(z,w)=	\varphi\left(g_1(\bR)(\bI-z\bU)^{-1}g_2(\bR)(\bI-w\bU)^{-1}\right)\label{Defofk}.
	\end{equation}
	
	\begin{theorem}\label{ThMainExpectation}
		For $z,w$ in some neighborhood of $0$ in $\C$  we have
		\begin{equation*}
			k(z,w)=k_1(z,w)+k_2(z,w),
		\end{equation*}
		with
		\begin{eqnarray*}
			k_1(z,w)&=&\varphi\left(g_1(\bR)g_2(\bR)(\bI-\od\bR)^{-1}(\bI-\odw\bR)^{-1}\right),\\
			k_2(z,w)&=&\frac{\left(w\od-z\odw\right)\left(\od-\odw\right)}{\left(\muz-\muw\right)(z-w)}A_1(z,w)A_2(z,w)\\
			&=&\frac{w\od-z\odw}{z-w}\frac{A_1(z,w)A_2(z,w)}{B(z,w)},
		\end{eqnarray*}
		where
		\begin{eqnarray*}
			A_k(z,w)&=&\varphi\left(\bR g_k(\bR)(\bI-\od\bR)^{-1}(\bI-\odw\bR)^{-1}\right),\ \ k=1,2,\\
			B(z,w)&=&\varphi\left(\bR (\bI-\od\bR)^{-1}(\bI-\odw\bR)^{-1}\right)=\frac{\muz-\muw}{\od-\odw}.
		\end{eqnarray*}
		
		Here $\od$ is the subordination function satisfying $\muz=M_{\bR}(\od)$.
	\end{theorem}
	
	\begin{remark}
		The function $k(z,w)$ is continuous  for $(z,w)\in (\C\setminus\R_+)^2$ and analytic in both $z$ and $w$. For fixed $w$ in some neighborhood of $0$ the function $k_1(z,w)+k_2(z,w)$ is meromorphic in $z\in \C\setminus\R_+$ and agrees with $k(z,w)$ in the neighborhood of $0$.
		Uniqueness of analytic continuation  shows that $k_1(\cdot,w)+k_2(\cdot,w)$ has only removable singularities and $k(z,w)=k_1(z,w)+k_2(z,w)$ holds for all $z\in \C\setminus\R_+$.  Fixing $z\in \C\setminus\R_+$ and repeating the argument shows that $k=k_1+k_2$ on $(\C\setminus\R_+)^2.$
	\end{remark}
	An immediate consequence of Theorem \ref{ThMainExpectation} and property \eqref{CondProp2} is the following Corollary.
	\begin{corollary}
		Under the assumptions of Theorem \ref{ThMainExpectation}
		\begin{multline*}
			\varphi\left((\bI-z\bU)^{-1}g(\bR)(\bI-w\bU)^{-1}\mid\bR\right)=g(\bR)(\bI-\od\bR)^{-1}(\bI-\odw\bR)^{-1}\\
			+\frac{w\od-z\odw}{z-w}\frac{A(z,w)}{B(z,w)}\bR(\bI-\od\bR)^{-1}(\bI-\odw\bR)^{-1},
		\end{multline*}
		where
		\begin{eqnarray*}
			A(z,w)&=&\varphi\left(\bR g(\bR)(\bI-\od\bR)^{-1}(\bI-\odw\bR)^{-1}\right),\\
			B(z,w)&=&\varphi\left(\bR (\bI-\od\bR)^{-1}(\bI-\odw\bR)^{-1}\right).
		\end{eqnarray*}
	\end{corollary}
	
	\subsection{Proof of Theorem \ref{ThMainExpectation}} The proof will be divided into a series of lemmas and propositions.
	Before we start  we introduce some notation. Let $h(\bR)$ belong to the unital  subalgebra generated by $\bR$. We define
	\begin{align}
		\eta^h_\bR(z)&=\sum_{k=0}^\infty\beta_{k+1}\left(h(\bR),\underbrace{\bR,\bR,\ldots,\bR}_{k}\right)z^k,\label{etahz}\\ 
		\eta^h_\bR(z,w)&=	\sum_{k,l\geq 0}\beta_{k+l+1}\left(\underbrace{\bR,\bR,...,\bR}_{l},h(\bR),\underbrace{\bR,\bR,...,\bR}_{k}\right)z^lw^k.\label{etahzw}
	\end{align}
	Both $\eta^h_\bR(z)$ and $\eta^h_\bR(z,w)$ are  well-defined near the origin. The function $\eta^h_\bR(z)$ was defined in \cite{szpojankowski2019conditional} and the authors showed a method of calculating this function when $h(\bR)$ is a holomorphic function of $\bR$. 
	
	The function $\eta^h_\bR(z,w)$ generalizes the function $\eta^h_\bR(z)$ in the sense that $\eta^h_\bR(z,0)=\eta^h_\bR(z)$. Since Boolean cumulants are invariant under reflections i.e. $\beta_n(\bX_1,\bX_2,\ldots,\bX_n)=\beta_n(\bX_n,\ldots,\bX_2,\bX_1)$ then it is clear that $\eta^h_\bR(z,w)=\eta^h_\bR(w,z)$. Lemma \ref{lematoetahzw} shows that $\eta^h_\bR(z,w)$ can be expressed in terms of $\eta_\bR^h(z)$, furthermore Lemma \ref{wzorynaety} shows that
	\begin{eqnarray*}
		\eta^h_\bR(z)&=&\frac{\varphi\left(h(\bR)(\bI-z\bR)^{-1}\right)}{\varphi\left((\bI-z\bR)^{-1}\right)},\\
		\eta^h_\bR(z,w)&=&\frac{\varphi\left(h(\bR)(\bI-z\bR)^{-1}(\bI-w\bR)^{-1}\right)}{\varphi\left((\bI-z\bR)^{-1}\right)\varphi\left((\bI-w\bR)^{-1}\right)},
	\end{eqnarray*}
	which allows us to extend both functions analytically for  $z,w\in\C\setminus\R_+$ and provides a way of calculating them at least when $h(\bR)$ is some nice function of $\bR$.
	
	We will also use the following notation for Boolean cumulants that will  frequently appear in the proof:
	\begin{eqnarray*}
		\beta_n(\bY)&=&\beta_n(\bY,\bY,...,\bY),\\
		y_n&=&\beta_{2n+1}\left(\bY,\bR,\bY,\ldots,\bR,\bY\right),\\
		r_n&=&\beta_{2n+1}\left(\bR,\bY,\bR,\ldots,\bY,\bR\right),\\
		s^{h}_{m,n}&=&\beta_{2m+2n+3}(\underbrace{\bY,\bR,\bY,...,\bY}_{2m+1},h(\bR),\underbrace{\bY,\bR,\bY,...,\bY}_{2n+1}),\\	t^{h}_{m,n}&=&\beta_{2m+2n+1}(\underbrace{\bR,\bY,...,\bR,\bY}_{2m},h(\bR),\underbrace{\bY,\bR,...,\bY,\bR}_{2n}).
	\end{eqnarray*}
	For  technical reasons it will be easier to work with
	\begin{equation*}
		f(z,w)=	\varphi\left(g_1(\bR)z\bU(\bI-z\bU)^{-1}g_2(\bR)w\bU(\bI-w\bU)^{-1}\right).\label{Defoff}
	\end{equation*} 
	Since  $(\bI-z\bU)^{-1}=\bI+z\bU(\bI-z\bU)^{-1}$, the difference $k(z,w)-f(z,w)$ is equal
	$$\varphi\left(g_1(\bR)g_2(\bR)\right)+\varphi\left(g_1(\bR)z\bU(\bI-z\bU)^{-1}g_2(\bR)\right)+\varphi\left(g_1(\bR)g_2(\bR)w\bU(\bI-w\bU)^{-1}\right)$$
	and can be easily calculated using \eqref{conditionalsubordination}. See the proof of Lemma \ref{lemmaf1} for more details.\newline
	For $z, w$ is some neighborhood of $0$ in $\C$ we can write
	\begin{equation*}
		f(z,w)=\sum_{m,n\geq1}\varphi\left(g_1(\bR)\bU^m g_2(\bR)\bU^n\right)z^mw^n.
	\end{equation*}
	For any $m,n\geq 1$ by traciality of $\varphi$ we have
	\begin{eqnarray*}
		\varphi\left(g_1(\bR)\bU^m g_2(\bR)\bU^n\right)&=&\varphi\left(g_1(\bR)\left(\bR^{\frac{1}{2}}\bY\bR^{\frac{1}{2}}\right)^mg_2(\bR)\left(\bR^{\frac{1}{2}}\bY\bR^{\frac{1}{2}}\right)^n\right)\\
		&=&\varphi\left(\underbrace{\bY\bR\bY\cdots\bR\bY}_{2m-1}h_2(\bR)\underbrace{\bY\bR\bY\cdots\bR\bY}_{2n-1}h_1(\bR)\right),
	\end{eqnarray*}
	where  $h_k(\bR)=\bR g_k(\bR)$ for $k=1,2$.
	Therefore $$f(z,w)=\sum_{m,n\geq1}\varphi\left(\underbrace{\bY\bR\bY\cdots\bR\bY}_{2m-1}h_2(\bR)\underbrace{\bY\bR\bY\cdots\bR\bY}_{2n-1}h_1(\bR)\right)z^mw^n.$$
	To evaluate $f(z,w)$ we apply formula \eqref{boolmain1} for the following two free collections of random variables
	\begin{equation*}
		\{\underbrace{\bR,\bR,...,\bR}_{m-1},h_2(\bR),\underbrace{\bR,\bR,...,\bR}_{n-1},h_1(\bR)\}\ \mbox{and}\ \{\underbrace{\bY,\bY,..,\bY}_{m+n}\}.
	\end{equation*}
	In order to apply this formula we have to consider two cases whether the sequence $1\leq j_1<\ldots <j_{k+1}=m+n$ appearing in  formula \eqref{boolmain1} contains $m$ or not. (This is due to the fact that $m$'th element of the first collection is $h_2(\bR)$.) Hence
	$$\varphi\left(\underbrace{\bY\bR\bY\cdots\bR\bY}_{2m-1}h_2(\bR)\underbrace{\bY\bR\bY\cdots\bR\bY}_{2n-1}h_1(\bR)\right)=S_{m,n}^{(1)}+S_{m,n}^{(2)},$$
	where  $S_{m,n}^{(1)}$ is a sum over all sequences $1\leq j_1<\ldots <j_{k+1}=m+n$ containing $m$ and $S_{m,n}^{(2)}$ contains all other elements.
	
	Explicitly we have
	\begin{equation}
		S_{m,n}^{(1)}=
		\sum_{l=0}^{m-1}\sum_{k=0}^{n-1}\varphi\left(h_1(\bR)h_2(\bR)\bR^{k+l}\right)
		\sum_{\substack{m_0+\cdots+m_l=m-1-l\\n_0+\cdots+n_k=n-1-k}}y_{m_0}y_{m_1}\cdots y_{m_l}\cdot y_{n_0}y_{n_1}\cdots y_{n_k},\label{s1mn}
	\end{equation}
	\begin{equation}
		S_{m,n}^{(2)}=
		\sum_{l=0}^{m-1}\sum_{k=0}^{n-1}\varphi\left(h_1(\bR)\bR^{k+l}\right)\sum_{\substack{m_0+\cdots+m_l=m-1-l\\n_0+\cdots+n_k=n-1-k}}y_{m_0}y_{m_1}\cdots y_{m_{l-1}}\cdot s^{h_2}_{m_l,n_0}\cdot y_{n_1}\cdots y_{n_k}.\label{s2mn}
	\end{equation}
	The indices $m_0,..,m_l,n_0,...,n_k$ appearing in the most inner sums are assumed to be non-negative integers.
	
	We can now  write $f(z,w)=\sum_{m,n\geq 1}S^{(1)}_{m,n}z^mw^n+\sum_{m,n\geq 1}S^{(2)}_{m,n}z^mw^n$. In the  next two lemmas we calculate both series.
	\begin{lemma}\label{lemmaf1}
		For $z,w$ in some neighborhood of $0$ in $\C$ the sum $$f_1(z,w)=\sum_{m,n\geq 1}S^{(1)}_{m,n}z^mw^n$$ is equal
		\begin{equation}
			f_1(z,w)=\od\odw\varphi\left(h_1(\bR)h_2(\bR)(\bI-\od\bR)^{-1}(\bI-\odw\bR)^{-1}\right).\label{f1formula}
		\end{equation}
		As a result  the function $k_1(z,w):=k(z,w)-f(z,w)+f_1(z,w)$ is equal
		\begin{equation}
			k_1(z,w)=\varphi\left(g_1(\bR)g_2(\bR)(\bI-\od\bR)^{-1}(\bI-\odw\bR)^{-1}\right).\label{k1formula}
		\end{equation}
		
	\end{lemma}
	\begin{proof}
		Using \eqref{s1mn}, changing variables and the order of summation several times yields
		\begin{eqnarray*}
			&&f_1(z,w)=\sum_{m,n\geq1}S^{(1)}_{mn}z^mw^n\\
			&=&\sum_{m,n\geq1}z^mw^n\sum_{l=0}^{m-1}\sum_{k=0}^{n-1}\varphi\left(h_1(\bR)h_2(\bR)\bR^{k+l}\right)
			\sum_{\substack{m_0+\cdots+m_l=m-1-l\\n_0+\cdots+n_k=n-1-k}}y_{m_0}y_{m_1}\cdots y_{m_l}\cdot y_{n_0}y_{n_1}\cdots y_{n_k}\\
			&=&\sum_{k,l\geq 0}\varphi\left(h_1(\bR)h_2(\bR)\bR^{k+l}\right)\sum_{\substack{m\geq l+1\\n\geq k+1}}z^mw^n\sum_{\substack{m_0+\cdots+m_l=m-1-l\\n_0+\cdots+n_k=n-1-k}}y_{m_0}y_{m_1}\cdots y_{m_l}\cdot y_{n_0}y_{n_1}\cdots y_{n_k}\\
			&=&\sum_{k,l\geq 0}\varphi\left(h_1(\bR)h_2(\bR)\bR^{k+l}\right)z^{l+1}w^{k+1}\sum_{\substack{m,n\geq 0}}z^mw^n
			\sum_{\substack{m_0+\cdots+m_l=m\\n_0+\cdots+n_k=n}}y_{m_0}y_{m_1}\cdots y_{m_l}\cdot y_{n_0}y_{n_1}\cdots y_{n_k}
		\end{eqnarray*}
		The inner sum i.e.
		$$\sum_{\substack{m,n\geq 0}}z^mw^n
		\sum_{\substack{m_0+\cdots+m_l=m\\n_0+\cdots+n_k=n}}y_{m_0}y_{m_1}\cdots y_{m_l}\cdot y_{n_0}y_{n_1}\cdots y_{n_k}$$
		is equal $B(z)^{l+1}B(w)^{k+1}$ where
		$$B(z)=\sum_{n=0}^\infty y_nz^n=\sum_{n=0}^\infty\beta_{2n+1}(\bY,\bR,\bY,...,\bR,\bY)z^n=\frac{\od}{z},$$
		where the last equality follows from \eqref{rowzw2}.   
		Hence \begin{eqnarray*}
			f_1(z,w)&=&\od\odw\sum_{k,l\geq 0}\varphi\left(h_1(\bR)h_2(\bR)\bR^{k+l}\right)\od^l\odw^k\\
			&=&\od\odw\varphi\left(h_1(\bR)h_2(\bR)(\bI-\od\bR)^{-1}(\bI-\odw\bR)^{-1}\right).
		\end{eqnarray*}
		This proves \eqref{f1formula}. 
		
		To prove \eqref{k1formula} note that  identity $(\bI-z\bU)^{-1}=\bI+z\bU(\bI-z\bU)^{-1}$ together with \eqref{conditionalsubordination} allows us to write  $k(z,w)-f(z,w)+f_1(z,w)$ as
		\begin{multline*}
			\varphi(g_1(\bR)g_2(\bR)\od\bR(\bI-\od\bR)^{-1})
			+\varphi(g_1(\bR)g_2(\bR)\odw\bR(\bI-\odw\bR)^{-1})\\
			+\varphi(g_1(\bR)g_2(\bR))+\od\odw\varphi\left(g_1(\bR)g_2(\bR)\bR^2(\bI-\od\bR)^{-1}(\bI-\odw\bR)^{-1}\right).\label{gminusf}
		\end{multline*}
		The last expression is equal exactly $k_1(z,w)$ from \eqref{k1formula}.
	\end{proof} 
	\begin{lemma}\label{lemmak2}
		For $z,w$ in some neighborhood of $0$ in $\C$ the sum  $k_2(z,w)=\sum_{m,n\geq 1}S^{(2)}_{m,n}z^mw^n$ is equal
		\begin{equation}
			k_2(z,w)=G(z,w)\varphi\left(h_1(\bR)(\bI-\od\bR)^{-1}(\bI-\odw\bR)^{-1}\right),\label{k2formula}
		\end{equation}
		where 
		\begin{equation*}
			G(z,w)=\sum_{m,n\geq 0}\beta_{2m+2n+3}\left(\underbrace{\bY,\bR,\bY,....,\bR,\bY}_{2m+1},h_2(\bR),\underbrace{\bY,\bR,\bY,....,\bR,\bY}_{2n+1}\right)z^{m+1}w^{n+1}.
		\end{equation*}
		
	\end{lemma}
	\begin{proof}
		The proof of \eqref{k2formula} is similar to the proof of Lemma \ref{lemmaf1}. We apply formula \eqref{s2mn} to get
		\begin{eqnarray*}
			&&f_2(z,w)=\sum_{m,n\geq1}S^{(2)}_{mn}z^mw^n\\
			&=&\sum_{m,n\geq1}z^mw^n\sum_{l=0}^{m-1}\sum_{k=0}^{n-1}\varphi\left(h_1(\bR)\bR^{k+l}\right)\sum_{\substack{m_0+\cdots+m_l=m-1-l\\n_0+\cdots+n_k=n-1-k}}y_{m_0}y_{m_1}\cdots y_{m_{l-1}}\cdot s^{h_2}_{m_l,n_0}\cdot y_{n_1}\cdots y_{n_k}\\
			&=&	\sum_{l,k\geq 0}\varphi\left(h_1(\bR)\bR^{k+l}\right)\sum_{\substack{m\geq l+1\\n\geq k+1}}z^mw^n\sum_{\substack{m_0+\cdots+m_l=m-1-l\\n_0+\cdots+n_k=n-1-k}}y_{m_0}y_{m_1}\cdots y_{m_{l-1}}\cdot s^{h_2}_{m_l,n_0}\cdot y_{n_1}\cdots y_{n_k}\\
			&=&\sum_{l,k\geq 0}\varphi\left(h_1(\bR)\bR^{k+l}\right)z^{l+1}w^{k+1}\sum_{m,n\geq 0}z^mw^n\sum_{\substack{m_0+\cdots+m_l=m\\n_0+\cdots+n_k=n}}y_{m_0}y_{m_1}\cdots y_{m_{l-1}}\cdot s^{h_2}_{m_l,n_0}\cdot y_{n_1}\cdots y_{n_k}\\
			&=&\sum_{l,k\geq 0}\varphi\left(h_1(\bR)\bR^{k+l}\right)z^{l}w^{k}G(z,w)B(z)^l B(w)^k,
		\end{eqnarray*}
		where  as in the previous proof	$$B(z)=\sum_{n=0}^\infty y_nz^n=\sum_{n=0}^\infty\beta_{2n+1}(\bY,\bR,\bY,...,\bR,\bY)z^n=\frac{\od}{z}$$ and
		\begin{eqnarray*}
			G(z,w)&=&zw\sum_{m,n\geq 0}s^{h_2}_{m,n}z^mw^n\\
			&=&\sum_{m,n\geq 0}\beta_{2m+2n+3}\left(\underbrace{\bY,\bR,\bY,....,\bR,\bY}_{2m+1},h_2(\bR),\underbrace{\bY,\bR,\bY,....,\bR,\bY}_{2n+1}\right)z^{m+1}w^{n+1}.
		\end{eqnarray*}
		Hence \begin{eqnarray*}
			k_2(z,w)&=&G(z,w)\sum_{l,k\geq 0}\varphi\left(h_1(\bR)\bR^{k+l}\right)\od^l\odw^l\\
			&=&G(z,w)\varphi\left(h_1(\bR)(\bI-\od\bR)^{-1}(\bI-\odw\bR)^{-1}\right).	
		\end{eqnarray*}
	\end{proof}
	
	The next step is to find the  closed form for the series \begin{equation}
		G(z,w)=\sum_{m,n\geq 1}\beta_{2m+2n+1}\left(\underbrace{\bY,\bR,\bY,....,\bR,\bY}_{2m-1},h(\bR),\underbrace{\bY,\bR,\bY,....,\bR,\bY}_{2n-1}\right)z^{m}w^{n}\label{defofG}.
	\end{equation}
	As always $h(\bR)$ is a fixed  element in the unital subalgebra generated by $\bR$. It turns out that $G(z,w)$ is closely related to the following double series
	
	\begin{equation}
		H(z,w)=\sum_{m,n\geq 0}\beta_{2m+2n+1}\left(\underbrace{\bR,\bY,....,\bR,\bY}_{2m},h(\bR),\underbrace{\bY,\bR,\bY,....,\bR}_{2n}\right)z^mw^n\label{defofH}
	\end{equation}
	so we consider both at the same time.
	
	We start with the following lemma which shows  that $G(z,w)$ can be expressed in terms of $H(z,w)$ and the subordination functions.
	
	\begin{lemma}
		For $z,w$ in some neighborhood of $0$ in $\C$ such that $\oj\neq\ojw$ \label{LemmaGbyH}
		\begin{equation*}
			G(z,w)=H(z,w)\frac{w\od-z\odw}{\oj-\ojw}
		\end{equation*}
		where $\oj,\od$ are the subordination functions satisfying $M_{\bR\bY}(z)=\muz=M_{\bY}(\oj)=M_{\bR}(\od)$.
	\end{lemma}
	\begin{proof}
		We start by applying formula \eqref{boolmain3} to  the following free collections of random variables
		\begin{equation*}
			\{\underbrace{\bY,\bY,..,\bY}_{m+n}\}	\ \mbox{and}\ \{\underbrace{\bR,\bR,...,\bR}_{m-1},h(\bR),\underbrace{\bR,\bR,...,\bR}_{n-1}\}.
		\end{equation*}
		Thus for  $m,n\geq 1$ we obtain
		\begin{multline*}
			\beta_{2m+2n-1}(\underbrace{\bY,\bR,\bY,....,\bR,\bY}_{2m-1},h(\bR),\underbrace{\bY,\bR,\bY,....,\bR,\bY}_{2n-1})\\
			=\sum_{l=1}^{m}\sum_{k=1}^{n}\beta_{k+l}(\bY)\sum_{\substack{m_1+m_2+...+m_l=m-l\\n_1+n_2+...+n_k=n-k}}r_{m_1}\cdots r_{m_{l-1}}\cdot t^h_{m_l,n_1}\cdot r_{n_2}\cdots r_{n_l}
		\end{multline*}
		where appropriate  products of $r$'s are assumed to be $1$ if either $k=1$ or $l=1$. Application of the above formula and standard power series manipulations yield
		\begin{eqnarray*}
			&&G(z,w)=\\
			&=&\sum_{m,n\geq 1}z^mw^n\sum_{l=1}^{m}\sum_{k=1}^{n}\beta_{k+l}(\bY)\sum_{\substack{m_1+m_2+...+m_l=m-l\\n_1+n_2+...+n_k=n-k}}r_{m_1}\cdots r_{m_{l-1}}\cdot t^h_{m_l,n_1}\cdot r_{n_2}\cdots r_{n_l}\\
			&=&\sum_{k,l\geq 1}\beta_{k+l}(\bY)\sum_{\substack{m=l\\n=k}}^\infty\left(\sum_{\substack{m_1+m_2+...+m_l=m-l\\n_1+n_2+...+n_k=n-k}}r_{m_1}\cdots r_{m_{l-1}}\cdot t^h_{m_l,n_1}\cdot r_{n_2}\cdots r_{n_l}\right)z^mw^n\\
			&=&\sum_{k,l\geq 1}\beta_{k+l}(\bY)z^lw^k\sum_{m,n\geq0}^\infty\left(\sum_{\substack{m_1+m_2+...+m_l=m\\n_1+n_2+...+n_k=n}}r_{m_1}\cdots r_{m_{l-1}}\cdot t^h_{m_l,n_1}\cdot r_{n_2}\cdots r_{n_l}\right)z^mw^n.
		\end{eqnarray*}
		Note that the inner "double" sum  equals
		$$A(z)^{l-1}A(w)^{k-1}H(z,w),$$ where 	by \eqref{rowzw1}
		$$A(z)=\sum_{n=0}^\infty r_nz^n=\sum_{n=0}^\infty\beta_{2n+1}(\bR,\bY,\bR,...,\bY,\bR)z^n=\frac{\oj}{z}$$
		and $H(z,w)=\sum_{m,n\geq0}t^{h}_{m,n}z^mw^n$ which is \eqref{defofH}. Hence
		\begin{eqnarray*}
			G(z,w)&=&	\sum_{l=1}^{\infty}\sum_{k=1}^{\infty}\beta_{k+l}(\bY)z^lw^kA(z)^{l-1}A(w)^{k-1}H(z,w)\\
			&=&zwH(z,w)\sum_{l=1}^{\infty}\sum_{k=1}^{\infty}\beta_{k+l}(\bY)\oj^{l-1}\omega_1(w)^{k-1}.
		\end{eqnarray*} 
		Now if $\oj\neq\ojw$ then 
		$$\sum_{l=1}^{\infty}\sum_{k=1}^{\infty}\beta_{k+l}(\bY)\oj^{k-1}\oj^{l-1}=\sum_{n=0}^\infty\beta_{n+2}(\bY)\frac{\oj^{n+1}-\omega_1(w)^{n+1}}{\oj-\omega_1(w)}.$$
		The rest is straightforward:
		$$\sum_{n=0}^\infty\beta_{n+2}(\bY)z^{n+1}=\frac{\sum_{n=2}^\infty\beta_n(\bY)z^n}{z}=\frac{\eta_\bY(z)}{z}-\varphi(\bY).$$
		Hence $$G(z,w)=zwH(z,w)\frac{\omega_1(w)\eta_\bY(\oj)-\oj\eta_\bY(\omega_1(w))}{\oj\omega_1(w)(\oj-\omega_1(w))}.$$
		To finish the proof we apply formula \eqref{usefulidentity} that can be written as  $\oj\od=\eta_\bY(\oj)z$ to get
		$$zw\left[\omega_1(w)\eta_\bY(\oj)-\oj\eta_\bY(\omega_1(w))\right]=\oj\ojw\left(w\od-z\odw\right).$$
	\end{proof}
	We will show now that similar calculation yields another formula that links $H(z,w)$ with $G(z,w)$. Before we state the lemma we will express
	\begin{eqnarray*}
		H(z,w)=\sum_{m,n\geq 	0}\beta_{2m+2n+1}\left(\underbrace{\bR,\bY,....,\bR,\bY}_{2m},h(\bR),\underbrace{\bY,\bR,\bY,....,\bR}_{2n}\right)z^mw^n,
	\end{eqnarray*}
	as $H(z,w)=H_1(z,w)+C(z)+C(w)-\varphi(h(\bR))$ 
	where
	\begin{eqnarray*}
		H_1(z,w)&=&\sum_{m,n\geq 1}\beta_{2m+2n+1}\left(\underbrace{\bR,\bY,....,\bR,\bY}_{2m},h(\bR),\underbrace{\bY,\bR,\bY,....,\bR}_{2n}\right)z^mw^n,\\
		C(z)&=&\sum_{m\geq0}\beta_{2m+1}\left(h(\bR),\underbrace{\bY,\bR,\bY,....,\bR}_{2m}\right)z^m.
	\end{eqnarray*}
	We do this because  it is easier to work with $H_1(z,w)$ rather than $H(z,w)$. The reason is that  $h(\bR)$ is always the middle term in Boolean cumulant appearing in $H_1(z,w)$ while it is not always the case for $H(z,w)$. The function $C(z)$ is equal
	\begin{equation}
		C(z)=\eta^h_\bR(\od),\label{Cofz}
	\end{equation} where $\eta_\bR^h(z)$ is given by the series form formula \eqref{etahz}. This was shown in \cite{szpojankowski2019conditional}   using \eqref{boolmain3} so we skip the details.
	\begin{lemma}\label{LemmaHbyG}
		For $z,w$ in some neighborhood of $0$ in $\C$ such that $\od\neq\odw$
		\begin{equation*}
			H(z,w)=\frac{G(z,w)}{zw}\frac{w\oj-z\ojw}{\od-\odw}+\eta^h_\bR(\od,\odw),
		\end{equation*}
		where $\eta^h_\bR(z,w)$ is given by formula \eqref{etahzw}.
	\end{lemma}
	\begin{proof}
		We start by applying formula \eqref{boolmain3} to $t^h_{m,n}$ with $m,n\geq 1$ to get
		\begin{equation*}
			t^h_{m,n}=\beta_{2m+2n+1}\left(\underbrace{\bR,\bY,....,\bR,\bY}_{2m},h(\bR),\underbrace{\bY,\bR,\bY,....,\bR}_{2n}\right)=E_{m,n}+F_{m,n}
		\end{equation*}
		where
		\begin{eqnarray*}
			E_{m,n}&=&\sum_{\substack{1\leq l\leq m\\1\leq k\leq n}}\beta_{k+l+1}\left(\underbrace{\bR,...,\bR}_{l},h(\bR),\underbrace{\bR,...,\bR}_{k}\right)\
			\sum_{\substack{m_1+\cdots+m_l=m-l\\n_1+\cdots+n_k=n-k}}y_{m_1}\cdots y_{m_l}\cdot y_{n_1}\cdots y_{m_k},\\
			F_{m,m}&=&\sum_{\substack{1\leq l\leq m\\1\leq k\leq n}}\beta_{k+l}(\bR)\sum_{\substack{m_1+m_2+...+m_l=m-l\\n_1+n_2+...+n_k=n-k}}y_{m_1}\cdots y_{m_{l-1}}\cdot s^h_{m_l,n_1}\cdot y_{n_2}\cdots y_{m_k}.
		\end{eqnarray*}
		Let us denote $$E(z,w)=\sum_{m,n\geq 1}E_{m,n}z^mw^n,\ \ \mbox{and}\ \  F(z,w)=\sum_{m,n\geq 1}F_{m,n}z^mw^n.$$ Consequently $H(z,w)=E(z,w)+F(z,w)+C(z)+C(w)-\varphi(h(\bR))$. 
		The  above formulas for $E_{m,n}, F_{m,n}$ allow us to evaluate $E(z,w)$ and  $F(z,w)$. One can show that
		\begin{eqnarray*}
			E(z,w)&=&\sum_{k,l\geq 1}\beta_{k+l+1}\left(\underbrace{\bR,\bR,...,\bR}_{l},h(\bR),\underbrace{\bR,\bR,...,\bR}_{k}\right)\od^l\odw^k,\\
			F(z,w)&=&	\frac{G(z,w)}{zw}\frac{w\oj-z\ojw}{\od-\odw},
		\end{eqnarray*}
		where $G(z,w)=\sum_{m,n\geq 0}s^h_{m,n}z^{m+1}w^{n+1}$ is the function from \eqref{defofG}. The proof of these identities is similar to the ones from Lemma \ref{LemmaGbyH} and we skip it.
		
		To finish the proof note that \eqref{Cofz} implies that $E(z,w)+C(z)+C(w)-\varphi(h(\bR))$ is exactly $\eta^h_\bR(\od,\odw)$.
	\end{proof}
	Combining Lemmas \ref{LemmaGbyH} and \ref{LemmaHbyG} yields the following Corollary.
	\begin{corollary}\label{GHexplicitformulas}
		The functions $G(z,w), H(z,w)$ given by formulas \eqref{defofG} and \eqref{defofH} are equal
		
		\begin{equation*}
			G(z,w)=\frac{\left(w\od-z\odw\right)\left(\od-\odw\right)}{\left(\eta_{\bU}(z)-\eta_{\bU}(w)\right)(z-w)}\eta^h_\bR(\od,\odw),
		\end{equation*}
		\begin{equation*}
			H(z,w)=\frac{(\oj-\ojw)\left(\od-\odw\right)}{\left(\eta_{\bU}(z)-\eta_{\bU}(w)\right)(z-w)}\eta^h_\bR(\od,\odw).
		\end{equation*} 
	\end{corollary}
	\begin{proof}
		Lemmas \ref{LemmaGbyH} and \ref{LemmaHbyG}  imply that $G(z,w)$ and $H(z,w)$ satisfy the following system of equations
		\begin{eqnarray*}
			\left\{\begin{array}{rl}
				G(z,w)&=H(z,w)\frac{w\od-z\odw}{\oj-\ojw},\\
				H(z,w)&=\frac{G(z,w)}{zw}\frac{w\oj-z\ojw}{\od-\odw}+\eta^h_\bR(\od,\odw).
			\end{array}\right.
		\end{eqnarray*}
		Hence $$G(z,w)\left(1-\frac{1}{zw}\frac{w\od-z\odw}{\oj-\ojw}\frac{w\oj-z\ojw}{\od-\odw}\right)=\frac{w\od-z\odw}{\oj-\ojw}\eta^h_\bR(\od,\odw).$$
		Simple algebra and \eqref{usefulidentity} i.e. $\oj\od=z\eta_{\bU}(z)$ show that the expression in the parenthesis is equal to
		$$\frac{(z-w)(\eta_{\bU}(z)-\eta_{\bU}(w))}{(\oj-\ojw)(\od-\odw)}.$$
		This ends the proof.
	\end{proof}

	In the next two lemmas we will show that  $\eta^h_\bR(\cdot,\cdot)$ can be expressed in terms of $\eta^h_\bR(\cdot)$ and $\eta_\bR(\cdot)$. We also  give explicit formulas for both functions.

	\begin{lemma}\label{lematoetahzw}
		For $z\neq w$ in some neighborhood of $0$ in $\C$ the following identity holds
		\begin{equation}
			\eta^h_\bR(z,w)=\frac{z\eta^h_\bR(z)-w\eta^h_\bR(w)}{z-w}+\frac{\eta_\bR(z)w\eta^h_\bR(w)-\eta_\bR(w)z\eta^h_\bR(z)}{z-w}.\label{lematoetach1}
		\end{equation}
		
	\end{lemma}
	\begin{proof}
		Let $g(\bR)$ be any element in the unital  subalgebra generated by $\bR$ and let $h(\bR)=\bR\cdot g(\bR)$. Then
		
		$$\eta^h_\bR(z,w)=	\sum_{k,l\geq 0}\beta_{k+l+1}\left(\underbrace{\bR,\bR,...,\bR}_{l},\bR g(\bR),\underbrace{\bR,\bR,...,\bR}_{k}\right)z^lw^k.$$
		Applying Proposition \ref{betanaproduktach} we get
		$$\beta_{k+l+1}\left(\underbrace{\bR,\bR,...,\bR}_{l},\bR g(\bR),\underbrace{\bR,\bR,...,\bR}_{k}\right)=\beta_{k+l+1+1}\left(\underbrace{\bR,\bR,...,\bR}_{l+1}, g(\bR),\underbrace{\bR,\bR,...,\bR}_{k}\right)+
		$$	
		$$+\beta_{l+1}\left(\bR,\bR,...,\bR\right)\cdot\beta_{k+1}\left( g(\bR),\underbrace{\bR,\bR,...,\bR}_{k}\right).$$
		Hence $\eta^h_\bR(z,w)=S_1+S_2$ where
		\begin{eqnarray*}
			S_1&=&\sum_{k,l\geq 0}\beta_{k+l+1+1}\left(\underbrace{\bR,\bR,...,\bR}_{l+1}, g(\bR),\underbrace{\bR,\bR,...,\bR}_{k}\right)z^lw^k\\
			&=&\frac{1}{z}\sum_{k\geq 0,l\geq 1}\beta_{k+l+1}\left(\underbrace{\bR,\bR,...,\bR}_{l}, g(\bR),\underbrace{\bR,\bR,...,\bR}_{k}\right)z^lw^k\\
			&=&\frac{1}{z}\eta^g_\bR(z,w)-\frac{1}{z}\eta^g_\bR(w),
		\end{eqnarray*}	
		and
		\begin{eqnarray*}
			S_2&=&\sum_{k,l\geq 0}\beta_{l+1}\left(\bR,\bR,...,\bR\right)\cdot\beta_{k+1}\left( g(\bR),\underbrace{\bR,\bR,...,\bR}_{k}\right)z^lw^k\\
			&=&\left(\sum_{l\geq 0}\beta_{l+1}\left(\bR,\bR,...,\bR\right)z^l\right)\sum_{k\geq 0}\beta_{k+1}\left( g(\bR),\underbrace{\bR,\bR,...,\bR}_{k}\right)w^k\\
			&=&\frac{1}{z}\eta_\bR(z)\eta^g_\bR(w).
		\end{eqnarray*}
		Therefore
		$$\eta^h_\bR(z,w)=\frac{1}{z}\eta^g_\bR(z,w)-\frac{1}{z}\eta^g_\bR(w)+\frac{1}{z}\eta_\bR(z)\eta^g_\bR(w).$$
		Since the functions $\eta^h_\bR(z,w)$ and $\eta^g_\bR(z,w)$ are symmetric in $z,w$ we get
		$$\frac{1}{w}\eta^g_\bR(z,w)-\frac{1}{w}\eta^g_\bR(z)+\frac{1}{w}\eta_\bR(w)\eta^g_\bR(z)=\frac{1}{z}\eta^g_\bR(z,w)-\frac{1}{z}\eta^g_\bR(w)+\frac{1}{z}\eta_\bR(z)\eta^g_\bR(w).$$
		After some algebra we get 
		\begin{equation}
			(z-w)\eta^g_\bR(z,w)=z\eta^g_\bR(z)-w\eta^g_\bR(w)+\eta_\bR(z)w\eta^g_\bR(w)-\eta_\bR(w)z\eta^g_\bR(z).\label{wzor1}
		\end{equation}
		This proves formula \eqref{lematoetach1} (for $g$ instead of $h$). 
	\end{proof}
	\begin{lemma}\label{wzorynaety}
		For $z,w$ in some neighborhood of $0$ in $\C$
		\begin{align}
			\eta^h_\bR(z)&=\frac{\varphi\left(h(\bR)(\bI-z\bR)^{-1}\right)}{\varphi\left((\bI-z\bR)^{-1}\right)},\label{wzornaetaz}\\ 
			\eta^h_\bR(z,w)&=\frac{\varphi\left(h(\bR)(\bI-z\bR)^{-1}(\bI-w\bR)^{-1}\right)}{\varphi\left((\bI-z\bR)^{-1}\right)\varphi\left((\bI-w\bR)^{-1}\right)}.\label{wzornaetazw}
		\end{align}
	\end{lemma}
	\begin{proof}
		For $z$ near the origin we have
		\begin{equation*}
			\varphi\left(h(\bR)(\bI-z\bR)^{-1}\right)=\sum_{n=0}^\infty\varphi\left(h(\bR)\bR^n\right)z^n.
		\end{equation*}
		By the definition of  Boolean cumulants we have
		\begin{eqnarray*}
			\varphi\left(h(\bR)\bR^n\right)&=&\sum_{k=1}^{n+1}\sum_{\stackrel{j_1+j_2+\cdots+j_k=n+1}{j_1,...,j_k\geq 1}}\beta_{j_1}(h(\bR),\bR,...,\bR)\prod_{m=2}^k\beta_{j_m}(\bR)\\
			&=&\sum_{k=1}^{n+1}\sum_{\stackrel{i_1+i_2+\cdots+i_k=n+1-k}{i_1,...,i_k\geq 0}}\beta_{i_1+1}(h(\bR),\bR,...,\bR)\prod_{m=2}^k\beta_{i_m+1}(\bR).
		\end{eqnarray*}
		Thus
		\begin{eqnarray*}
			\varphi\left(h(\bR)(\bI-z\bR)^{-1}\right)&=&\sum_{n=0}^\infty\sum_{k=1}^{n+1}\sum_{\stackrel{i_1+i_2+\cdots+i_k=n+1-k}{i_1,...,i_k\geq 0}}\beta_{i_1+1}(h(\bR),\bR,...,\bR)\prod_{m=2}^k\beta_{i_m+1}(\bR)z^n\\
			&=&\sum_{k=1}^{\infty}\sum_{n=k-1}^\infty\sum_{\stackrel{i_1+i_2+\cdots+i_k=n+1-k}{i_1,...,i_k\geq 0}}\beta_{i_1+1}(h(\bR),\bR,...,\bR)\prod_{m=2}^k\beta_{i_m+1}(\bR)z^n\\
			&=&\sum_{k=1}^{\infty}\sum_{n=0}^\infty z^{k-1}\sum_{\stackrel{i_1+i_2+\cdots+i_k=n}{i_1,...,i_k\geq 0}}\beta_{i_1+1}(h(\bR),\bR,...,\bR)\prod_{m=2}^k\beta_{i_m+1}(\bR)z^n\\
			&=&\sum_{k=1}^{\infty}z^{k-1}\eta^h_{\bR}(z)\left(\frac{\eta_\bR(z)}{z}\right)^{k-1}=\frac{\eta^h_{\bR}(z)}{1-\eta_{\bR}(z)}.
		\end{eqnarray*}
		This proves \eqref{wzornaetaz} since $\frac{1}{1-\eta_{\bR}(z)}=1+M_{\bR}(z)=\varphi\left((\bI-z\bR)^{-1}\right)$. The formula \eqref{wzornaetazw} follows now from combining \eqref{wzornaetaz} with formula \eqref{lematoetach1}.
	\end{proof}
	The following Corollary ends the proof of Theorem \ref{ThMainExpectation}.
	\begin{corollary}
		For $z,w$ in some neighborhood of $0$ in $\C$ the sum  $k_2(z,w)$ from Lemma \ref{lemmak2} is equal
		\begin{eqnarray*}
			k_2(z,w)&=&\frac{\left(w\od-z\odw\right)\left(\od-\odw\right)}{\left(\muz-\muw\right)(z-w)}A_1(z,w)A_2(z,w)\\
			&=&\frac{w\od-z\odw}{z-w}\frac{A_1(z,w)A_2(z,w)}{B(z,w)},
		\end{eqnarray*}
		where
		\begin{eqnarray*}
			A_k(z,w)&=&\varphi\left(\bR g_k(\bR)(\bI-\od\bR)^{-1}(\bI-\odw\bR)^{-1}\right),\\
			B(z,w)&=&\varphi\left(\bR (\bI-\od\bR)^{-1}(\bI-\odw\bR)^{-1}\right).
		\end{eqnarray*}
		
	\end{corollary}
	\begin{proof}
		Recall that $h_k(\bR)=\bR g_k(\bR)$. 		From Lemma \ref{lemmak2} and Corollary \ref{GHexplicitformulas} we see that
		
		$$k_2(z,w)=\frac{\left(w\od-z\odw\right)\left(\od-\odw\right)}{\left(\eta_{\bU}(z)-\eta_{\bU}(w)\right)(z-w)}\eta^{h_2}_\bR(\od,\odw)A_1(z,w).$$
		From formula \eqref{wzornaetazw} we have
		$$\eta^{h_2}_\bR(\od,\odw)=\frac{A_2(z,w)}{(\muz+1)(\muw+1)}.$$ 
		To finish the proof we note that $$\eta_{\bU}(z)-\eta_{\bU}(w)=\frac{\muz-\muw}{(\muz+1)(\muw+1)}$$
		and
		\begin{eqnarray*}
			\muz-\muw&=&M_{\bR}(\od)-M_{\bR}(\odw)\\
			&=&(\od-\od)\varphi\left(\bR (\bI-\od\bR)^{-1}(\bI-\odw\bR)^{-1}\right)\\
			&=&(\od-\od)B(z,w).
		\end{eqnarray*}

	\end{proof}
	
	We will also an state the explicit formula for $k(z,w)$ in the special case  when $g(\bR)=\bR(\bI-\bR)^{-1}$. This will be important later when we consider characterization theorems.
	\begin{proposition}\label{Funkcjak}
		Let $\bR$ and $\bY$ be free,  self-adjoint and  positive random variables such that $\bR<\bI$. Denote $g(\bR)=\bR(\bI-\bR)^{-1}$ and $\bU=\bR^{\frac{1}{2}}\bY\bR^{\frac{1}{2}}$. 
		Then the expectation $$k(z,w)=\varphi\left(g(\bR)(\bI-w\bU)^{-1}g(\bR)(\bI-z\bU)^{-1}\right),$$
		is equal
		\begin{equation}
			k(z,w)=k_1(z,w)+k_2(z,w),\label{EqFunkcjaF}
		\end{equation} where
		\begin{eqnarray*}
			k_1(z,w)&=&\frac{1}{\od-\odw}\left(\od\frac{\muz-\varphi(g(\bR))}{(\od-1)^2}-\odw\frac{\muw-\varphi(g(\bR))}{(\odw-1)^2}\right)+\\
			&&\frac{1}{(\od-1)(\odw-1)}\left(\varphi(g(\bR))+\varphi(g(\bR)^2)\right),
		\end{eqnarray*}
		$$k_2(z,w)=\frac{\left(w\od-z\odw\right)}{\left(\muz-\muw\right)(\od-\odw)(z-w)}\left(\frac{\muz-\varphi(g(\bR))}{\od-1}-\frac{\muw-\varphi(g(\bR))}{\odw-1}\right)^2.$$
	\end{proposition}
	\begin{proof}
		We need to evaluate $k_1(z,w)$ and $k_2(z,w)$  from  Theorem \ref{ThMainExpectation}.
		
		The function $k_1(z,w)$ satisfies
		\begin{equation*}
			k_1(z,w)=\varphi\left(g(\bR)^2(\bI-\od\bR)^{-1}(\bI-\odw\bR)^{-1}\right).
		\end{equation*}
		We start by writing
		$$(\bI-\od\bR)^{-1}(\bI-\odw\bR)^{-1}=\frac{1}{\od-\odw}\left(\od(\bI-\od\bR)^{-1}-\odw(\bI-\odw\bR)^{-1}\right).$$
		Hence
		$$k_1(z,w)=\frac{1}{\od-\odw}\left(\varphi\left(\od g(\bR)^2(\bI-\od\bR)^{-1}\right)-\varphi\left(\odw g(\bR)^2(\bI-\odw\bR)^{-1}\right)\right).$$
		The  identity
		$\tfrac{zx^2}{(1-x)^2(1-zx)}=\tfrac{z}{(z-1)^2}\left(\tfrac{zx}{1-zx}-\tfrac{x}{1-x}\right)-\tfrac{z}{z-1}\left(\tfrac{x}{1-x}+\tfrac{x^2}{(1-x)^2}\right)$
		implies that
		\begin{equation*}
			\varphi\left(\od g(\bR)^2(\bI-\od\bR)^{-1}\right)=\frac{\od}{(\od-1)^2}\left(\muz-\varphi(g(\bR))\right)-\frac{\od}{\od-1}\left(\varphi(g(\bR))+\varphi(g(\bR)^2)\right).
		\end{equation*}
		After some algebra we get
		\begin{eqnarray*}
			k_1(z,w)&=&\frac{1}{\od-\odw}\left(\od\frac{\muz-\varphi(g(\bR))}{(\od-1)^2}-\odw\frac{\muw-\varphi(g(\bR))}{(\odw-1)^2}\right)+\\
			&&\frac{1}{(\od-1)(\odw-1)}\left(\varphi(g(\bR))+\varphi(g(\bR)^2)\right).
		\end{eqnarray*}
		For $k_2(z,w)$ we have
		\begin{eqnarray*}
			k_2(z,w)=\frac{\left(w\od-z\odw\right)\left(\od-\odw\right)}{\left(\muz-\muw\right)(z-w)}A(z,w)^2,
		\end{eqnarray*}
		where $$A(z,w)=\varphi\left(\bR g(\bR)(\bI-\od\bR)^{-1}(\bI-\odw\bR)^{-1}\right).$$
		As earlier, after  using the formula
		$$(\bI-\od\bR)^{-1}(\bI-\odw\bR)^{-1}=\frac{1}{\od-\odw}\left(\od(\bI-\od\bR)^{-1}-\odw(\bI-\odw\bR)^{-1}\right)$$
		and the identity
		$\frac{zx^2}{(1-x)(1-zx)}=\frac{zx}{(z-1)(1-zx)}-\frac{zx}{(z-1)(1-x)}$
		we get
		\begin{eqnarray*}
			\varphi\left(\od \bR g(\bR)(\bI-\od\bR)^{-1}\right)&=&\frac{\muz}{\od-1}-\frac{\od}{\od-1}\varphi(\bR(\bI-\bR)^{-1})\\
			&=&\frac{\muz}{\od-1}-\frac{1}{\od-1}\varphi(\bR(\bI-\bR)^{-1})-\varphi(\bR(\bI-\bR)^{-1})\\
			&=&\frac{\muz-\varphi(g(\bR))}{\od-1}-\varphi(g(\bR)).
		\end{eqnarray*}
		Hence
		\begin{equation*}
			A(z,w)=\left[\frac{1}{\od-\odw}\left(\frac{\muz-\varphi(g(\bR))}{\od-1}-\frac{\muw-\varphi(g(\bR))}{\odw-1}\right)\right].
		\end{equation*}
		This ends the proof.
	\end{proof}
	\section{Main problem: the characterization theorems}
	Throughout this section we assume we are given two  free,  self-adjoint, positive and  non-degenerated variables $\bX,\bY$ in a $W^*$-probability space. We also assume  $\bU,\bV$  are defined by
	\begin{equation}\label{defofuandv}
		\left\{\begin{array}{rl}
			\bU&=(\bI+\bX)^{-\frac{1}{2}}\bY(\bI+\bX)^{-\frac{1}{2}},\\
			\bV&=(\bI+\bU)^{\frac{1}{2}}\bX(\bI+\bU)^{\frac{1}{2}}.
		\end{array}\right.
	\end{equation}
	We consider three sets of regression conditions
	\begin{equation*}
		\left\{\begin{array}{rl}
			\varphi\left(\bV\mid \bU\right)&=a\bI,\\
			\varphi\left(\bV^{-1}\mid \bU\right)&=c\bI,
		\end{array}\right.\ \ \mbox{or}\ \ 	\left\{\begin{array}{rl}
			\varphi\left(\bV\mid \bU\right)&=a\bI,\\
			\varphi\left(\bV^{2}\mid \bU\right)&=b\bI,
		\end{array}\right.\ \ \mbox{or}\ \ \left\{\begin{array}{rl}
			\varphi\left(\bV^{-1}\mid \bU\right)&=c\bI,\\
			\varphi\left(\bV^{-2}\mid \bU\right)&=d\bI,
		\end{array}\right.
	\end{equation*}
	where $a,b,c,d\in\R$ are some constants. The goal is to  show that any of  these three sets of conditions implies that $\bX,\bY$ have free-Kummer and free-Poisson distribution respectively.
	
	As earlier we  denote $\bR=(\bI+\bX)^{-1}$ which  is equivalent to $\bX=\bR^{-1}-1$. Thus we can express $\bU$ as $\bU=\bR^{\frac{1}{2}}\bY\bR^{\frac{1}{2}}$.  
	Note that the first and the third set of conditions involves $\bV^{-1}$ and therefore we require to assume that $\bX$ is strictly positive. In this case $\bR<\bI$ and $\bX^{-1}=\bR(\bI-\bR)^{-1}$.
	
	As before  $\oj,\od $ denote the subordination functions such that $\muz=M_{\bY}(\oj)=M_{\bR}(\od)$.
	
	Before we state our characterization theorem we show in the next few lemmas that each considered condition of the form $\varphi\left(\bV^k\mid \bU\right)=c_k\bI$ implies certain algebraic equation involving $\od$, $\muz$ and various constants. It is worth to note that for $k=1,-1$ these equations are  derived using only the technique of subordination for free multiplicative convolution. For  $k=2,-2$ relying solely on subordination does not work and derivation of these equations is based on  Theorem  \ref{ThMainExpectation}.
	
	Since  any two considered conditions give us two equations with two unknown functions, in theory we should be able to find them and obtain various transforms for considered random variables and thus determine their distribution. This is done in Subsection \ref{SekcjaRozkłady}.  See  Proposition \ref{kummerlemma} and  Lemmas \ref{DistofY} and  \ref{DistofX}  for details.

	\subsection{Functional equations implied by regression conditions}
	\begin{lemma}\label{lemmaeq1}
		Assume that
		\begin{equation}
			\varphi\left(\bV\mid \bU\right)=a\bI, \label{reg1} 
		\end{equation}
		for some constant $a$. Then  the identity 
		\begin{equation}
			(\od-1)M_{\bU}(z)+\od=\frac{z}{z+1}\left(a M_\bU(z)+a-\varphi(\bX)\right)\label{HVeq1}
		\end{equation}
		holds for all $z\in\C\setminus\R_+$.
	\end{lemma}
	\begin{proof}
		We start by rewriting \eqref{reg1} as
		\begin{equation}
			\varphi\left(\bX\mid\bU\right)=a(\bI+\bU)^{-1}.\label{r1}
		\end{equation}
		Multiplying both sides by $(\bI-z\bU)^{-1}$ and applying expectation gives
		\begin{equation}
			\varphi\left(\bX(\bI-z\bU)^{-1}\right)=a\varphi\left((\bI+\bU)^{-1}(\bI-z\bU)^{-1}\right).\label{r2}
		\end{equation}
		It is easy to check that for $z\neq -1$
		\begin{equation*}
			(1+\bU)^{-1}(\bI-z\bU)^{-1}=\frac{1}{z+1}(\bI+\bU)^{-1}+\frac{z}{z+1}z\bU(\bI-z\bU)^{-1}+\frac{z}{z+1}\bI.
		\end{equation*}
		Hence the right hand side equals
		\begin{equation}
			a\varphi\left((1+\bU)^{-1}(\bI-z\bU)^{-1}\right)=\frac{\varphi(\bX)}{z+1}+\frac{az}{z+1}M_\bU(z)+\frac{az}{z+1},
		\end{equation}
		since \eqref{r1} implies that $a\varphi((\bI+\bU)^{-1})=\varphi(\bX)$.
		
		By the subordination property (formula \eqref{conditionalsubordination}) the left hand side of \eqref{r2} is equal
		\begin{eqnarray*}
			\varphi\left(\bX(\bI-z\bU)^{-1}\right)&=&\varphi\left(\bX\varphi\left((\bI-z\bU)^{-1}\mid\bR\right)\right)\\
			&=&\varphi\left((\bR^{-1}-\bI)(\bI-\od \bR)^{-1}\right)\\
			&=&\varphi\left(\bR^{-1}+(\od-1)(\bI-\od\bR)^{-1}\right)\\
			&=&\varphi(\bX)+1+(\od-1)(1+M_{\bR}(\od))\\
			&=&\varphi(\bX)+\od+(\od-1)M_{\bU}(z).
		\end{eqnarray*}
		Hence \eqref{r2} yields
		\begin{equation}
			\varphi(\bX)+\od+(\od-1)M_{\bU}(z)=\frac{\varphi(\bX)}{z+1}+\frac{az}{z+1}M_\bU(z)+\frac{az}{z+1},
		\end{equation}
		which is equivalent to \eqref{HVeq1}.
	\end{proof}
	\begin{lemma}\label{lemmaeq3}
		Assume that
		\begin{equation}
			\varphi\left(\bV^{-1}\mid \bU\right)=c\bI, \label{reg3} 
		\end{equation}
		for some constant $c$. Then  the identity
		\begin{equation}
			\frac{z}{\od-1}\left(M_{\bU}(z)-\varphi(\bX^{-1})\right)=cz+c(z+1)M_{\bU}(z)\label{HVeq3}
		\end{equation}
		holds for all $z\in\C\setminus\R_+$.
	\end{lemma}
	\begin{proof}

		The condition $$\varphi\left(\bV^{-1}\mid \bU\right)=c\bI,$$
		is equivalent to
		\begin{equation}
			\varphi\left(\bX^{-1}\mid \bU\right)=c(\bI+\bU).\label{varphiofU}
		\end{equation}
		We multiply both sides by $z(\bI-z\bU)^{-1}$ and apply the state  $\varphi$ to get
		\begin{equation*}
			z\varphi\left(\bX^{-1}(\bI-z\bU)^{-1}\right)=cz\varphi\left((\bI+\bU)(\bI-z\bU)^{-1}\right).
		\end{equation*}
		We now evaluate both sides. The right hand side is equal
		\begin{eqnarray*}
			cz\varphi\left((\bI+\bU)(\bI-z\bU)^{-1}\right)	&=&cz\varphi\left(\bI+z\bU(\bI-z\bU)^{-1}\right)+c\varphi\left(z\bU(\bI-z\bU)^{-1}\right)\\
			&=&cz+cz\muz+c\muz\\
			&=&cz+c(z+1)\muz.
		\end{eqnarray*}
		For the  left hand side we apply formula \eqref{conditionalsubordination} to get
		\begin{eqnarray*}
			z\varphi\left(\bX^{-1}(\bI-z\bU)^{-1}\right)&=&z\varphi\left(\bX^{-1}\varphi((\bI-z\bU)^{-1}\mid \bR)\right)\\
			&=&z\varphi\left(\bR(\bI-\bR)^{-1}(1-\od\bR)^{-1}\right).
		\end{eqnarray*}
		From the identity
		$$\frac{x}{(1-x)(1-zx)}=\frac{1}{(z-1)}\frac{zx}{(1-zx)}-\frac{1}{(z-1)}\frac{x}{(1-x)},$$ we  see that
		\begin{eqnarray*}
			z\varphi\left(\bX^{-1}(\bI-z\bU)^{-1}\right)&=&\frac{z}{\od-1}\left(\muz-\varphi(\bR(\bI-\bR)^{-1})\right)\\
			&=&\frac{z}{\od-1}\left(\muz-\varphi(\bX^{-1})\right).
		\end{eqnarray*}
		
	\end{proof}
	\begin{lemma}\label{lemmaeq4}
		Assume that
		\begin{equation}
			\varphi\left(\bV^{-2}\mid \bU\right)=d\bI, \label{reg4} 
		\end{equation}
		for some constant $d$. Then  the following identity holds
		\begin{equation*}
			k(z,-1)=d+d\left(1+\frac{1}{z}\right)\muz,
		\end{equation*}
		for all $z\in\C\setminus\R_{+}$,	where $k(z,w)$ is given by the formula \eqref{EqFunkcjaF}.	
	\end{lemma}
	\begin{proof}
		We start by rewriting the condition \eqref{reg4} as
		\begin{equation}
			\varphi\left(\bX^{-1}(\bI+\bU)^{-1}\bX^{-1} \mid\bU\right)=d(\bI+\bU).\label{stala-p}
		\end{equation}
		Multiplying  both sides by $(\bI-z\bU)^{-1}$ and applying the state $\varphi$ to both sides gives 
		\begin{equation*}
			\varphi\left(\bX^{-1}(\bI+\bU)^{-1}\bX^{-1}(\bI-z\bU)^{-1}\right)=d\varphi\left((\bI+\bU)(\bI-z\bU)^{-1}\right).
		\end{equation*}
		The middle term is exactly  $k(z,-1)$ by  formula  \eqref{Defofk} which  defines the  function $k$. For the term on the right hand side we have
		\begin{equation*}
			d\varphi\left((\bI+\bU)(\bI-z\bU)^{-1}\right)=d+d\left(1+\frac{1}{z}\right)\muz.
		\end{equation*}
		This ends the proof of the Lemma.
	\end{proof}
	\begin{lemma}\label{lemmaeq2}
		Assume that
		\begin{equation}
			\varphi\left(\bV^2\mid \bU\right)=b\bI, \label{reg2} 
		\end{equation}
		for some constant $b$. Denote $D(z)=(\od-1)\muz+\od$. Then  the identity
		\begin{multline}
			\varphi(\bX^2)+\varphi(\bY)\varphi(\bX)+\od\varphi(\bX)+(\od-1)D(z)+\frac{D(z)^2}{z\muz}\\=\frac{b\varphi((\bI+\bU)^{-1})}{z+1}+\frac{bz}{z+1}(\muz+1)+\varphi(\bY)\frac{D(z)}{\muz}\label{HVeq2} 
		\end{multline}
		holds for all $z\in\C\setminus\R_{+}$.
	\end{lemma}
	\begin{proof}
		Since $\bV^2=(\bI+\bU)^{\frac{1}{2}}\bX(\bI+\bU)\bX(\bI+\bU)^{\frac{1}{2}}$, the condition  \eqref{reg2} is equivalent to
		\begin{equation}
			\varphi\left(\bX(\bI+\bU)\bX\mid\bU\right)=b(\bI+\bU)^{-1}.\label{r3}
		\end{equation}
		We multiply both sides by $(\bI-z\bU)^{-1}$ and apply $\varphi$ to get
		\begin{equation}
			\varphi\left(\bX^2(\bI-z\bU)^{-1}\right)+\varphi\left(\bX\bU\bX(\bI-z\bU)^{-1}\right)=b\varphi\left((\bI+\bU)^{-1}(\bI-z\bU)^{-1}\right).\label{reg2eq1}
		\end{equation}
		The right hand side was calculated in the previous Lemma and  is equal 
		\begin{equation*}
			b\varphi\left((\bI+\bU)^{-1}(\bI-z\bU)^{-1}\right)=	\frac{b\varphi((\bI+\bU)^{-1})}{z+1}+\frac{bz}{z+1}(\muz+1).
		\end{equation*}
		The first term of \eqref{reg2eq1} can be calculated using formula \eqref{conditionalsubordination} as follows
		\begin{eqnarray*}
			\varphi\left(\bX^2(\bI-z\bU)^{-1}\right)&=&\varphi\left(\bX^2\varphi\left((\bI-z\bU)^{-1}\mid\bR\right)\right)\\
			&=&\varphi\left((\bR^{-1}-\bI)^2(\bI-\od \bR)^{-1}\right)\\
			&=&\varphi\left((\bR^{-1}-\bI)^2+\od(\bR^{-1}-1)+(\od-1)\bI+(\od-1)^2(\bI-\od\bR)^{-1}\right)\\
			&=&\varphi(\bX^2)+\od\varphi(\bX)+\od-1+(\od-1)^2(1+\muz)\\
			&=&\varphi(\bX^2)+\od\varphi(\bX)+(\od-1)D(z),
		\end{eqnarray*}
		The middle term  of \eqref{reg2eq1} is equal 
		\begin{equation*}
			\varphi\left(\bX\bU\bX(\bI-z\bU)^{-1}\right)=\frac{\partial }{\partial w}k(z,w)\Bigr|_{w=0},
		\end{equation*}
		where
		$k(z,w)=\varphi\left(\bX(\bI-w\bU)^{-1}\bX(\bI-z\bU)^{-1}\right)$. Since $\bX=g(\bR)=\bR^{-1}-1$ by Theorem \ref{ThMainExpectation} we have $k(z,w)=k_1(z,w)+k_2(z,w)$ where
		\begin{eqnarray*}
			k_1(z,w)&=&\varphi\left((\bR^{-1}-\bI)^2(\bI-\od\bR)^{-1}(\bI-\odw\bR)^{-1}\right),\\
			k_2(z,w)&=&\frac{\left(w\od-z\odw\right)\left(\od-\odw\right)}{\left(\muz-\muw\right)(z-w)}\varphi\left((\bI-\bR)(\bI-\od\bR)^{-1}(\bI-\odw\bR)^{-1}\right)^2.
		\end{eqnarray*}
		Since $\omega_2'(0)=\varphi(\bY)$ (see formula \eqref{rowzw2} for example) we conclude that
		\begin{eqnarray*}
			\frac{\partial}{\partial w} k_1(z,w)\Bigr|_{w=0}&=&\omega_2'(0)\varphi\left((\bR^{-1}-\bI)^2(\bI-\od\bR)^{-1}\bR\right)\\
			&=&\varphi(\bY)\varphi\left(-\frac{1}{\od}\bI+\bR^{-1}+\frac{(\od-1)^2}{\od}(\bI-\od\bR)^{-1}\right)\\
			&=&\varphi(\bY)\left(-\frac{1}{\od}+\varphi(\bX)+1+\frac{(\od-1)^2}{\od}(1+\muz)\right)\\
			&=&\varphi(\bY)\varphi(\bX)+\frac{\varphi(\bY)}{\od}(\od-1)D(z).
		\end{eqnarray*}
		The function $w\od-z\odw$ is zero when $w=0$ therefore the product rule for derivatives implies that
		\begin{eqnarray*}
			\frac{\partial}{\partial w} k_2(z,w)\Bigr|_{w=0}&=&(\od-\varphi(\bY)z)\frac{\od}{z\muz}\varphi\left((\bI-\bR)(\bI-\od\bR)^{-1}\right)^2.
		\end{eqnarray*}
		Next we have $$ \varphi\left((\bI-\bR)(\bI-\od\bR)^{-1}\right)=1+\muz-\frac{1}{\od}\muz=\frac{D(z)}{\od}.$$
		Consequently
		\begin{equation*}
			\frac{\partial}{\partial w} k_2(z,w)\Bigr|_{w=0}=\frac{\od-\varphi(\bY)z}{z\od\muz}D(z)^2.
		\end{equation*}
		After some algebra we get
		\begin{equation*}
			\varphi\left(\bX\bU\bX(\bI-z\bU)^{-1}\right)=\varphi(\bY)\varphi(\bX)+\frac{D(z)^2}{z\muz}-\varphi(\bY)\frac{D(z)}{\muz}.
		\end{equation*}
		Combining and rearranging all terms gives \eqref{HVeq2}.
	\end{proof}
	
	\subsection{Determination of a distribution from functional equations} \label{SekcjaRozkłady}
	In the proofs of all three characterizations we will encounter similar equations that will  allow us to determine the  distributions of $\bX$ and $\bY$. In order to not repeat ourselves we will write down the common part of the  proof in a form of few lemmas. We start with the following proposition  that characterizes  free-Kummer distribution as the one for which its Cauchy-Stieltjes transform is  a solution of a  certain quadratic equation. (See also Remark \ref{EqForGremak}.)

	\begin{proposition}\label{kummerlemma}
		Let  $G=G(z)$ be a Cauchy-Stieltjes transform of a positive random variable $\bZ$. Suppose $G$ satisfies the following equation
		\begin{equation}
			z(z+1)G^2(z)-(\gamma z(z+1)-(\alpha-1) (z+1)+\beta z)G(z)+\gamma z+\delta=0,\label{kummereq}
		\end{equation}
		for some $\alpha,\gamma>0$, $\beta,\delta\in\R$. If either $\beta\geq 0$ or $\alpha>1$  then $\delta$ is uniquely determined by $\alpha,\beta,\gamma$  and $\bZ$ has the  free-Kummer distribution $\mathcal{K}(\alpha,\beta,\gamma)$.
		
	\end{proposition}
	\begin{proof}
		Let us denote 
		\begin{eqnarray*}
			p(z)&=&\gamma z(z+1)-(\alpha-1) (z+1)+\beta z\\
			&=&\gamma z^2-(\alpha-1-\beta-\gamma)z-(\alpha-1).
		\end{eqnarray*}
		Let also $H=H(z)$ be the Cauchy-Stieltjes transform of $\mathcal{K}(\alpha,\beta,\gamma)$. Then $H$ solves the equation \eqref{kummereq} with the same $\alpha,\beta,\gamma$ but possibly different $\delta$  say
		\begin{equation*}
			z(z+1)H^2(z)-p(z)H(z)+\gamma z+\delta_1=0.
		\end{equation*}
		Subtracting the above equation from \eqref{kummereq} gives the following identity
		\begin{equation}
			(G(z)-H(z))\left[z(z+1)(G(z)+H(z))-p(z)\right]=\delta_1-\delta,
		\end{equation}
		which holds for $z\in\C\setminus \R_+$ since both $G$ and $H$ are  Cauchy-Stieltjes transforms of a positive random variable. 
		
		Let us denote $f(z)=z(z+1)(G(z)+H(z))-p(z)$.  	Note that the  function $f(x)$ takes only real values for $x\in (-\infty,0)$. Since $p(x)$ is positive when $x\to-\infty$ and both $G$ and $H$ are negative on $(-\infty,0)$ we conclude that $f(x)$ is also negative when $x\to-\infty$.
		
		Since $f(-1)=-p(-1)=\beta$ and
		$$\limsup\limits_{x\to 0^-}f(x)=\limsup\limits_{x\to 0^-}\underbrace{x(x+1)(G(x)+H(x))}_{\textrm{positive on}\ (-1,0)}-p(0)\geq \alpha-1$$
		we conclude that when either $\beta\geq 0$ or $\alpha>1$ there exists $x_0\in(-\infty,0)$ such that $f(x_0)=0$. This implies	$\delta=\delta_1$ and consequently 
		\begin{equation}
			(G(z)-H(z))f(z)=0 \ \ \mbox{on}\ \C\setminus \R_+.\label{HalboG}
		\end{equation}
		Since $f(z)$ is non-zero analytic function we conclude that $G=H$.
	\end{proof}
	Using \eqref{M-to-G} we can reformulate the previous Proposition in terms of a moment transform.
	\begin{corollary}\label{KummmerMomentTransform}
		Assume the function $M=M(z)$ is a moment transform of a positive random variable $\bZ$ which satisfies the equation
		\begin{equation*}
			z(z+1)M^2(z)+((\alpha+1)z (z+1)-\gamma (z+1)-\beta z)M(z)+\alpha z(z+1)+(\delta-\gamma-\beta)z=0,
		\end{equation*}
		for some $\alpha,\gamma>0$, $\beta,\delta\in\R$. If either $\beta\geq 0$ or $\alpha>1$  then $\delta$ is uniquely determined by $\alpha,\beta,\gamma$  and $\bZ$ has the  free-Kummer distribution $\mathcal{K}(\alpha,\beta,\gamma)$. 
	\end{corollary}
	\begin{lemma}\label{DistofY}
		Suppose \begin{equation}
			\gamma\od=z\muz+\alpha z,\label{DistofYEq1}
		\end{equation}
		for some $\gamma>0$ and $\alpha\in \R$.
		Then $\alpha>0$ and $\bY\sim \nu\left(\alpha,1/\gamma\right)$.
	\end{lemma}
	
	\begin{proof}
		By  differentiating \eqref{DistofYEq1} at $z=0$ we  get $\gamma\varphi(\bY)=\alpha>0$.
		From \eqref{usefulidentity} we get $z=\frac{M_{\bU}(z)+1}{M_{\bU}(z)}\oj\od$.
		Plugging this into \eqref{DistofYEq1} allows us to cancel out $\od$ and yields
		\begin{equation*}
			\gamma=(M_{\bU}(z)+1)\oj+\alpha\tfrac{M_{\bU}(z)+1}{M_{\bU}(z)}\oj.
		\end{equation*}
		This gives $$\oj=\frac{\muz}{(\muz+1)\left[1/\gamma\muz+\alpha/\gamma\right]}.$$
		Since $M_{\bU}(z)=M_{\bY}(\oj)$ this implies
		\begin{equation*}
			M^{\langle-1\rangle}_\bY(z)=\frac{z}{(z+1)\left[1/\gamma\cdot z+\alpha/\gamma\right]}.
		\end{equation*}
		Hence $$S_{\bY}(z)=\frac{1}{1/\gamma\cdot z+\alpha/\gamma}.$$
		Since $S$-transform uniquely determines the distribution we get   $\bY\sim \nu\left(\alpha,1/\gamma\right)$.
	\end{proof}
	
	From the above proof it is  easy to see that  if  $\bY\sim \nu\left(\alpha,1/\gamma\right)$ then \eqref{DistofYEq1} holds. We will use this fact in the proof of the following Lemma.
	
	\begin{lemma}\label{DistofX}
		Assume $M=\muz$ satisfies 
		\begin{equation}
			z(z+1)M^2+((\alpha+1)z (z+1)-\gamma (z+1)-\beta z)M+\alpha z(z+1)+\rho z=0,\label{DistofXeq1}
		\end{equation}
		for some $\alpha,\beta,\gamma>0$ then $\bU\sim\mathcal{K}(\alpha,\beta,\gamma)$. Moreover if  $\bY\sim \nu\left(\alpha,1/\gamma\right)$ then
		$\bX\sim \mathcal{K}(\beta,\alpha,\gamma)$.
	\end{lemma}
	\begin{proof}
		
		The first claim follows directly from Corollary \ref{KummmerMomentTransform}. For the second claim we note that the assumption $\bY\sim \nu\left(\alpha,1/\gamma\right)$ implies  identity \eqref{DistofYEq1} i.e. $$z=\frac{\gamma\od}{\muz+\alpha}.$$
		We plug this into \eqref{DistofXeq1} and after some easy algebra we get
		\begin{multline*}
			(\od-1)\muz^2+\left[\gamma\od(\od-1)-(\beta-1)\od+\alpha(\od-1)\right]\muz\\
			+\gamma\od(\od-1)+(\alpha+\rho)\od=0.
		\end{multline*}
		Since $\muz=M_{\bR}(\od)$ we conclude that $\mrz$ satisfies
		\begin{equation}
			(z-1)\mrz^2+\left[\gamma z(z-1)-(\beta-1)z+\alpha(z-1)\right]\mrz  +\gamma z(z-1)+(\alpha+\rho)z=0.\label{DistofXEq2}
		\end{equation}
		Now since $\bR=(\bI+\bX)^{-1}$ then $$\mrz=\varphi\left(z\bR(\bI-z\bR)^{-1}\right)=z\varphi\left((\bX+(1-z)\bI)^{-1}\right)=-zG_{\bX}(z-1).$$
		Thus we can rewrite the  equation \eqref{DistofXEq2} for $\mrz$ as an equation for $\gx$.
		After some algebra we get
		$$z(z+1)\gx^2-\left(\gamma z(z+1)-(\beta-1)(z+1)+\alpha z\right)\gx+\gamma z+\alpha+\gamma+\rho=0.$$
		The lemma follows now from Proposition \ref{kummerlemma}.
	\end{proof}
	
	\subsection{The first characterization}
	\begin{theorem}\label{th1}
		Let $\bX,\bY$ be free, non-degenerated and self-adjoint random variables. Let also $\bY$ be positive and $\bX$ strictly positive. If $\bU,\bV$ are defined by \eqref{defofuandv} and 
		\begin{equation*}	\varphi\left(\bV\mid \bU\right)=a\bI,  \end{equation*}
		\begin{equation*}\varphi\left(\bV^{-1}\mid \bU\right)=c\bI,\end{equation*}
		for some constants $a$ and $c$, then $ac>1$ and $\bX$ has free-Kummer distribution   $\mathcal{K}\left(\frac{ac}{ac-1},\frac{\lambda}{ac-1},\frac{c}{ac-1}\right)$ and $\bY$ has  free Poisson distribution $\nu\left(\frac{\lambda}{ac-1}, \frac{ac-1}{c}\right)$, where $\lambda$ is some  positive constant.
	\end{theorem}
	
	\begin{proof}
		We have  $a=\varphi(\bV)>0$, $c=\varphi(\bV^{-1})>0$ then $ac>1$ by Cauchy-Schwarz inequality. (The inequality is strict because $V$ is non-degenerated and $\varphi$ is faithful.)
		Under the  assumptions of   Theorem \ref{th1}, Lemmas \ref{lemmaeq1} and \ref{lemmaeq3} imply the following system of equations
		\begin{equation} \label{ch1sys1} \left\{\begin{array}{rcl}
				(\od-1)M_{\bU}(z)+\od&=&\frac{z}{z+1}\left(a M_\bU(z)+a-p\right),\\
				\frac{cz}{z+1}(\od-1)+c(\od-1)M_{\bU}(z)&=&\frac{z}{z+1}\left(M_{\bU}(z)-q\right),
			\end{array} \right.
		\end{equation}
		where we denoted $p=\varphi(\bX), q=\varphi(\bX^{-1})$.
		Multiplying the first equation by $c$ and subtracting   second one from it gives
		\begin{equation*}
			c\frac{z+\od}{z+1}=(ac-1)\frac{z}{z+1}M_\bU(z)+(\lambda+c)\frac{z}{z+1},
		\end{equation*}
		where $\lambda=c(a-p)+q-c$. Hence
		\begin{equation}
			c\od=(ac-1)zM_\bU(z)+\lambda z.	\label{th1eq1}
		\end{equation}
		Lemma \ref{DistofY} implies that $\lambda>0$ and  $\bY\sim \nu\left(\frac{\lambda}{ac-1}, \frac{ac-1}{c}\right)$.
		
		We plug $\od$ from \eqref{th1eq1} into the second equation of \eqref{ch1sys1} to get the following equation
		\begin{equation*}
			(ac-1)z(z+1)M^2_{\bU}(z)+\left[(ac-1+\lambda)z(z+1)-c(z+1)-acz\right]M_{\bU}(z)+\lambda z(z+1)+c(p-a)z=0.
		\end{equation*}
		We recognize here the equation \eqref{DistofXeq1} from Lemma \ref{DistofX} with parameters $\alpha=\frac{\lambda}{ac-1}$, $\beta=\frac{ac}{ac-1}$,  $\gamma=\frac{c}{ac-1}$ and $\rho=\frac{c(p-a)}{ac-1}$. Clearly $\alpha,\beta,\gamma>0$. 
		Lemma \ref{DistofX} implies	$\bX\sim \mathcal{K}\left(\frac{ac}{ac-1},\frac{\lambda}{ac-1},\frac{c}{ac-1}\right)$.
		
	\end{proof}
	\begin{remark}\label{Uwagaoalfie}
		Note that the assumption that $\bX$ is invertible is reflected by having $\alpha=\frac{ac}{ac-1}>1$. This is also visible in Theorems \ref{th2}. In Theorem  \ref{th3} where $\bX$ is not necessarily invertible, we have $\alpha=\frac{a^2}{b-a^2}$ which can  a priori be either greater, equal, or less than $1$.
	\end{remark}
	\begin{remark}
		The characterization of free-Kummer and free-Poisson distributions from Theorem \ref{th1} was already  proved in \cite{PiliszekFreeKummer} using different method. The proof in that article relies on determining the distribution of $\bY$ and $\bU$ first and then using HV property (Theorem \ref{FreeHVproperty}) to determine the distribution of $\bX$. Therefore this argument requires  to assume that $\bY$ or equivalently $\bU$ is invertible. This assumption  is not needed in the proof of Theorem \ref{th1}.
	\end{remark}
	
	\subsection{The second characterization}
	
	\begin{theorem}\label{th2}
		Let $\bX,\bY$ be free, non-degenerated and self-adjoint random variables. Let also $\bY$ be positive and $\bX$ strictly positive. If $\bU,\bV$ are defined by \eqref{defofuandv} and 
		\begin{equation*}	\varphi\left(\bV^{-1}\mid \bU\right)=c\bI,  \end{equation*}
		\begin{equation*}\varphi\left(\bV^{-2}\mid \bU\right)=d\bI,\end{equation*}
		for some constants $c$ and $d$, then $d>c^2$ and $\bX$ has free-Kummer distribution  $\mathcal{K}\left(\frac{d}{d-c^2},\frac{\lambda}{d-c^2},\frac{c^3}{d-c^2}\right)$ and $\bY$ has  free Poisson distribution $\nu\left(\frac{\lambda}{d-c^2},\frac{d-c^2}{c^3}\right)$, where $\lambda$ is some  positive constant.
	\end{theorem}
	\begin{proof}
		Let us denote  $a=\varphi(\bX^{-1})$, $b=\varphi(\bX^{-2})$, $K=M_\bU(-1)$, $p=\omega_2(-1)$.\newline
		Under the  assumptions of   Theorem \ref{th2}, Lemmas \ref{lemmaeq3} and \ref{lemmaeq4} imply the following equations
		\begin{align}
			&\frac{M_{\bU}(z)-a}{\od-1}=c+c\left(1+\frac{1}{z}\right)M_{\bU}(z),\label{Ch2Eq1}\\
			&k(z,-1)=d\left(1+\frac{1}{z}\right)\muz+d,\label{Ch2Eq2}
		\end{align}
		where $k(z,w)=k_1(z,w)+k_2(z,w)$ is a function from Proposition \ref{Funkcjak}. Explicitly we have  
		\begin{eqnarray*}
			k_1(z,-1)&=&\frac{a+b}{(\od-1)(p-1)}+\frac{1}{\od-p}\left(\od\tfrac{\muz-a}{(\od-1)^2}-p\tfrac{K-a}{(p-1)^2}\right),\\
			k_2(z,-1)&=&\frac{-\left(\od+pz\right)}{\left(\muz-K\right)(\od-p)(z+1)}\left(\frac{\muz-a}{\od-1}-\tfrac{K-a}{p-1}\right)^2.
		\end{eqnarray*}
		
		The inequality $d>c^2$ follows from Cauchy-Schwarz inequality, non-degeneracy of $\bV$ and faithfulness of $\varphi$. The constants $K$ and $p$  can be expressed in terms of $a,b,c,d$. From \eqref{varphiofU} we get $\varphi(\bU)=\frac{a}{c}-1$ and \eqref{Ch2Eq1} implies that  $\frac{K-a}{p-1}=c$ i.e. $K=a+c(p-1)$. Taking the limit $z\to 0$ in \eqref{Ch2Eq2} yields
		\begin{equation*}
			d\varphi(\bU)+d=k_1(0,-1)=-\frac{a+b}{p-1}+\frac{K-a}{(p-1)^2}=\frac{c-a-b}{p-1}.
		\end{equation*}
		Consequently
		\begin{equation}
			p-1=\frac{c(c-a-b)}{ad},\ \ \ K=a+c(p-1).\label{Ch2constants}
		\end{equation}
		
		The next step it to simplify \eqref{Ch2Eq2}. From \eqref{Ch2Eq1} and \eqref{Ch2constants} we get
		\begin{equation*}
			k_2(z,-1)=\frac{-c^2\left(\od+pz\right)(z+1)\musz}{\left(\muz-K\right)(\od-p)z^2},
		\end{equation*}
		and 
		\begin{equation*}
			k_1(z,-1)=c\left(1+\frac{1}{z}\right)\muz\frac{\od}{(\od-1)(\od-p)}+\frac{a+b-c}{p-1}\frac{1}{\od-1}.
		\end{equation*}
		Hence \eqref{Ch2Eq2} is equivalent to
		\begin{equation*}
			c\left(1+\frac{1}{z}\right)\muz\frac{\od}{(\od-1)(\od-p)}+\frac{a+b-c}{p-1}\frac{1}{\od-1}+k_2(z,-1)=d\left(1+\frac{1}{z}\right)\muz+d.
		\end{equation*}
		By \eqref{Ch2constants}  we have $\frac{a+b-c}{p-1}=-\frac{ad}{c}$ and this  allows us to rewrite the last equation as
		\begin{equation}
			c\left(1+\frac{1}{z}\right)\muz\frac{\od}{(\od-1)(\od-p)}+k_2(z,-1)=d\left(1+\frac{1}{z}\right)\muz+d+\frac{ad}{c}\frac{1}{\od-1}.\label{Ch2Eq3}
		\end{equation}
		Using \eqref{Ch2Eq1} we see that the right hand side of \eqref{Ch2Eq3}  equals
		$$\frac{d}{c}\frac{\muz-a}{\od-1}+\frac{ad}{c}\frac{1}{\od-1}=\frac{d}{c}\frac{\muz}{\od-1}.$$
		Consequently \eqref{Ch2Eq3} takes the following equivalent form
		\begin{equation*}
			c\left(1+\frac{1}{z}\right)\muz\frac{\od}{(\od-1)(\od-p)}-\frac{c^2\left(\od+pz\right)(z+1)\musz}{\left(\muz-K\right)(\od-p)z^2}=\frac{d}{c}\frac{\muz}{\od-1}.
		\end{equation*}
		We cancel out $\muz$ and multiply both sides by $z(\od-1)(\od-p)$ which gives
		\begin{equation}
			c(z+1)\od-\frac{c^2(\od+pz)(z+1)(\od-1)\muz}{(\muz-K)z}=\frac{d}{c}z(\od-p).\label{Ch2Eq4}
		\end{equation}
		We rewrite  \eqref{Ch2Eq1} as
		\begin{eqnarray*}
			c(z+1)(\od-1)\muz&=&z(\muz-a)-cz(\od-1)\\
			&=&z\left[(\muz-K)-c(\od-p)\right],
		\end{eqnarray*}
		and apply it to the middle term of \eqref{Ch2Eq4} to get
		\begin{eqnarray*}
			c(z+1)\od-c(\od+zp)+\frac{c^2(\od+pz)(\od-p)}{\muz-K}=\frac{d}{c}z(\od-p).
		\end{eqnarray*}
		Since $c(z+1)\od-c(\od+zp)=cz(\od-p)$ we get
		\begin{equation*}
			cz+\frac{c^2(\od+pz)}{\muz-K}=\frac{d}{c}z.
		\end{equation*}
		After some simple algebra we arrive at
		\begin{equation}
			c^3\od=(d-c^2)z\muz+\lambda z,\label{keyequation2}
		\end{equation}
		where
		\begin{equation}
			\lambda=c^2K-c^3p-dK=ac^2-c^3-dK, \label{Th2Lambda}
		\end{equation} since by \eqref{Ch2constants} $cp=K+c-a$. Lemma \ref{DistofY} implies $\lambda>0$ and 
		$\bY\sim \nu\left(\frac{\lambda}{d-c^2},\frac{d-c^2}{c^3}\right)$.
		
		Combining \eqref{Ch2Eq1} with \eqref{keyequation2} gives the following equation for $\muz$:
		\begin{multline*}
			(d-c^2)z(z+1)M^2_\bU(z)+\left[(\lambda+d-c^2)z(z+1)-c^3(1+z)-dz\right]\muz+\\
			+\lambda z(z+1)+(c^2a-c^3-\lambda)z=0.
		\end{multline*}
		We recognize here the equation \eqref{DistofXeq1} from Lemma \ref{DistofX} with parameters
		$\alpha=\frac{\lambda}{d-c^2}$, $\beta=\frac{d}{d-c^2}$, $\gamma=\frac{c^3}{d-c^2}$ and $\rho=\frac{c^2a-c^3-\lambda}{d-c^2}$. Obviously $\alpha,\beta,\gamma>0$ 
		so  Lemma \ref{DistofX} implies that $\bX\sim \mathcal{K}\left(\tfrac{d}{d-c^2},\tfrac{\lambda}{d-c^2},\tfrac{c^3}{d-c^2}\right)$.
	\end{proof}
	\subsection{The third characterization}
	\begin{theorem}\label{th3}
		Let $\bX,\bY$ be free, positive, non-degenerated and self-adjoint random variables.  If $\bU,\bV$ are defined by \eqref{defofuandv} and 
		\begin{equation*}	\varphi\left(\bV\mid \bU\right)=a\bI,  \end{equation*}
		\begin{equation*}\varphi\left(\bV^2\mid \bU\right)=b\bI,\end{equation*}
		for some constants $a$ and $b$, then $b>a^2$ and $\bX$ has free-Kummer distribution   $\mathcal{K}\left(\frac{a^2}{b-a^2},\frac{a\lambda}{b-a^2},\frac{a}{b-a^2}\right)$ and $\bY$ has  free Poisson distribution $\nu\left(\frac{\lambda a}{b-a^2},\frac{b-a^2}{a}\right)$, where $\lambda$ is some  positive constant.
	\end{theorem}
	\begin{proof}
		The inequality $b>a^2$ follows from Cauchy-Schwarz inequality.
		Under the assumptions of the Theorem \ref{th3}, Lemmas \ref{lemmaeq1} and \ref{lemmaeq2} imply the following identities
		\begin{equation}
			D(z):=(\od-1)\muz+\od=\frac{z}{z+1}\left(a M_\bU(z)+a-p\right),\label{ch3Eq1}
		\end{equation}
		\begin{multline}
			q+\lambda p+p\od+(\od-1)D(z)+\frac{D(z)^2}{z\muz}=\\=\frac{bp}{a(z+1)}+\frac{bz}{z+1}(\muz+1)+\lambda\frac{D(z)}{\muz}.\label{ch3Eq2}
		\end{multline}
		where we denote $p=\varphi(\bX),q=\varphi(\bX^2),\lambda=\varphi(\bY)>0$. We also used the fact that \eqref{r1} implies that $\varphi((\bI+\bU)^{-1})=\frac{p}{a}$.
		
		Before we continue we will  show that  $q$ can be expressed by $a,b,p$ and $\lambda$. From $\eqref{r1}$ and freeness of $\bR$ and $\bY$ we obtain that 
		\begin{equation}
			a=\varphi(\bX)+\varphi(\bX\bU)=p+\varphi((\bI-\bR)\bY)=p+\lambda(1-\varphi(\bR)).\label{ch3Eq5}
		\end{equation}
		Similarly \eqref{r3} implies
		$$\frac{bp}{a}=\varphi(\bX^2)+\varphi(\bX^2\bU)=q+\varphi((\bR^{-1}-\bI)(\bI-\bR)\bY)=q+\lambda(p-1+\varphi(\bR)).$$
		Combining both identities we obtain
		\begin{equation}
			\frac{bp}{a}=q+\lambda p-(a-p).\label{Ch3Stale}
		\end{equation}
		The next step is to simplify \eqref{ch3Eq2}. We start by noting that
		\begin{equation*}
			\frac{bp}{a(z+1)}+\frac{bz}{z+1}(\muz+1)=\frac{b}{a}\left[\frac{p+pz}{z+1}+\frac{z}{z+1}(a\muz+a-p)\right]=\frac{bp}{a}+\frac{b}{a}D(z).
		\end{equation*}
		The above identity together with \eqref{Ch3Stale} allows us to write \eqref{ch3Eq2} in the following equivalent  form 
		\begin{equation}
			z(a-p)+pz\od+z(\od-1)D(z)+\frac{D(z)^2}{\muz}=\frac{b}{a}zD(z)+\lambda z\frac{D(z)}{\muz}.\label{ch3Eq3}
		\end{equation}
		Now we note that \eqref{ch3Eq1} can be written in the following two forms
		\begin{equation*}
			z(a-p)=(z+1)D(z)-az\muz\ \ \mbox{or}\ \ z(D(z)+p)=az\muz-D(z).
		\end{equation*}
		Consequently the left hand side of \eqref{ch3Eq3} is equal
		\begin{eqnarray*}
			&&z(a-p)+pz\od+z(\od-1)D(z)+D(z)\left(\od-1+\frac{\od}{\muz}\right)\\
			&=&z(a-p)+z(p+D(z))\od-zD(z)+(\od-1)D(z)+\frac{\od}{\muz}D(z)\\
			&=&(z+1)D(z)-az\muz+(az\muz-D(z))\od-zD(z)+(\od-1)D(z)+\frac{\od}{\muz}D(z)\\
			&=&\frac{\od}{\muz}D(z)+az\left[\od(\muz+1)-\muz\right]=\frac{\od}{\muz}D(z)+azD(z).
		\end{eqnarray*}
		and thus \eqref{ch3Eq3} is equivalent to
		\begin{equation*}
			\frac{\od}{\muz}D(z)+azD(z)=\frac{b}{a}zD(z)+\lambda z\frac{D(z)}{\muz}.
		\end{equation*}
		We can cancel out $D(z)$ as it is non-zero analytic function and get the following equation
		\begin{equation}
			\od=\frac{b-a^2}{a}z\muz+\lambda z.\label{ch3eq4}
		\end{equation}
		From Lemma \ref{DistofY} we get that $\bY\sim \nu\left(\frac{\lambda a}{b-a^2},\frac{b-a^2}{a}\right)$.
		Plugging $\od$ from \eqref{ch3eq4} into \eqref{ch3Eq1} yields the following equation for $\muz$
		\begin{multline*}
			(b-a^2)z(z+1)\muz^2+\left[(b-a^2+a\lambda)z(z+1)-a(1+z)-a^2z\right]\muz+\\a\lambda z(z+1)+a(p-a)z=0.
		\end{multline*} 
		This is the equation  \eqref{DistofXeq1} from Lemma \ref{DistofX} with parameters
		$\alpha=\frac{a\lambda}{b-a^2}$, $\beta=\frac{a^2}{b-a^2}$, $\gamma=\frac{a}{b-a^2}$ and $\rho=\frac{a(p-a)}{b-a^2}$. 
		Since $\alpha, \beta,\gamma>0$ we conclude from Lemma \ref{DistofX} that 	$\bX\sim \mathcal{K}\left(\frac{a^2}{b-a^2},\frac{a\lambda}{b-a^2},\frac{a}{b-a^2}\right)$.
	\end{proof}
	\section{Appendix}
	\subsection{Free-Kummer distribution $\mathcal{K}(1,\beta,\gamma)$}
	We will explain now in more details the definition of  free-Kummer distribution for $\alpha=1$.
	In that case the system of equations \eqref{freeKummerab} takes the following form
	\begin{equation*}
		\left\{\begin{array}{lc}
			\gamma+\frac{\beta}{\sqrt{(a+1)(b+1)}}&=0,\\
			\gamma\frac{a+b}{2}+\beta-\frac{\beta}{\sqrt{(a+1)(b+1)}}&=2.
		\end{array}\right.
	\end{equation*}
	and can easily be solved explicitly. A solution exists when $\beta<0$ and assuming $a<b$ we have
	\begin{equation*}
		a=-1+\tfrac{1}{\gamma}(1-\sqrt{1-\beta})^2 \ \ \ \mbox{and}\ \ \ b=-1+\tfrac{1}{\gamma}(1+\sqrt{1-\beta})^2 .
	\end{equation*}
	Form this we easily see that $a>0$ if and only if $1-\beta> (1+\sqrt{\gamma})^2$. 
	Thus  we define $\mathcal{K}(1,\beta,\gamma)$ with  $1-\beta>(1+\sqrt{\gamma})^2$ to be a probability measure with the  density \eqref{KummerDensity} that takes now the form 	\begin{equation}
		\frac{\gamma}{2\pi}\frac{\sqrt{(x-a)(b-x)}}{x+1}\mathbbm{1}_{(a,b)}(x).
	\end{equation}
	It's easy to check that the above function is the density of $X-1$ where $X$ has the free-Poisson distribution $\nu(1-\beta,1/\gamma)$.
	
	When $1-\beta\leq (1+\sqrt{\gamma})^2$ we defined   $\mathcal{K}(1,\beta,\gamma)$
	to be a probability measure that has the density
	\begin{equation}
		f_{\beta,\gamma}(x)=\frac{1}{2\pi}\sqrt{x(b-x)}\left(\tfrac{\sigma}{x}-\tfrac{\beta}{(1+x)\sqrt{b+1}}\right)\mathbbm{1}_{(0,b)}(x),\label{kummerdensityApp}
	\end{equation}
	where  $\sigma=\gamma+\frac{\beta}{\sqrt{b+1}}$ and $b$ is the unique positive solution of \begin{equation}
		\gamma\tfrac{b}{2}+\beta-\tfrac{\beta}{\sqrt{(b+1)}}=2.\label{EqforbApp}
	\end{equation} 
	(Existence of such solution is easy to see. Uniqueness follows from the fact the RHS is a strictly increasing function of $b$ for $\beta\geq 0$ and strictly convex function of $b$ and equal $0$ at $b=0$ for $\beta<0$.) We will show now that \eqref{kummerdensityApp} is a density function if and only if $1-\beta\leq (1+\sqrt{\gamma})^2$.
	
	For $b>0$ we have
	$$\int_{0}^b\frac{\sqrt{x(b-x)}}{x}dx=\frac{\pi b}{2}\ \ \mbox{and}\ \ \ \int_{0}^b\frac{\sqrt{x(b-x)}}{x+1}dx=\frac{\pi (\sqrt{1+b}-1)^2}{2}$$
	and we can  check that $f_{\beta,\gamma}$ integrates to $1$ whenever $b$ is a solution of \eqref{EqforbApp}. To finish, we will show the following fact.
	\begin{fact}
		$f_{\beta,\gamma}$ is non-negative  if and only if $1-\beta\leq (1+\sqrt{\gamma})^2$.
	\end{fact}
	\begin{proof}

		Note that 
		$$\tfrac{\sigma}{x}-\tfrac{\beta}{(1+x)\sqrt{b+1}}=\tfrac{\sigma}{x}-\tfrac{\sigma-\gamma}{x+1}=\tfrac{\sigma+\gamma x}{x(x+1)}\geq 0\  \mbox{for}\ x\in(0,b)\iff \sigma\geq 0,$$
		so the goal is to show $\sigma\geq 0\iff \beta\geq -\gamma -2\sqrt{\gamma}$. It's obvious that $\sigma\geq 0$ when $\beta\geq 0$ or $\beta<0$ but $\beta+\gamma\geq 0$. So we can assume $\beta<-\gamma<0$.
		
		Let us denote $u=\gamma-\sigma=-\frac{\beta}{\sqrt{b+1}}$. Note that $u>0$ since we assumed $\beta<-\gamma<0$. The condition $b>0$ is equivalent to  $u=-\frac{\beta}{\sqrt{b+1}}<-\beta$. Since $b=\frac{\beta^2}{u^2}-1$ we conclude that $u$ is the unique solution of 
		\begin{equation}
			\tfrac{\gamma\beta^2}{2u^2}-\tfrac{\gamma}{2}+\beta+u=2,
		\end{equation}
		such that  $u\in(0,-\beta)$. If we denote $h(u)=\frac{\gamma\beta^2}{2u^2}-\frac{\gamma}{2}+\beta+u$ then $\lim\limits_{u\to 0^-}h(u)=+\infty$ and $h(-\beta)=0$. Let us consider two cases
		\begin{itemize}
			\item If $\beta\in[-\gamma-2\sqrt{\gamma},-\gamma)$ then $\gamma\in (0,-\beta)$ and $(\beta+\gamma)^2\leq 4\gamma$. Observe that 
			$$h(\gamma)=\tfrac{\beta^2}{2\gamma}-\tfrac{\gamma}{2}+\beta+\gamma=\tfrac{(\beta+\gamma)^2}{2\gamma}\leq 2.$$
			Since there is only one solution to $h(x)=2$ for $x\in(0,-\beta)$ we conclude that $u=\gamma-\sigma\leq \gamma$ which proves that $\sigma\geq 0$.
			\item Similarly if  $\beta<-\gamma-2\sqrt{\gamma}$ then still $\gamma\in(0,-\beta)$ but now  $(\beta+\gamma)^2>2\gamma$ so $h(\gamma)>2$ and hence $u>\gamma$ and  $\sigma<0$.
		\end{itemize}
		
	\end{proof}
	\begin{remark}
		If $\beta=-\gamma-2\sqrt{\gamma}$ then $h(\gamma)=2$ and therefore $\sigma=0$. For $\beta\in(-\gamma-2\sqrt{\gamma},-\gamma)$ we have $h(\gamma)<2$ and thus $\sigma>0$ (also trivially $\sigma>0$ for $\beta\geq-\gamma)$.
	\end{remark}	
	
	\subsection*{Acknowledgment} The author thanks  K. Szpojankowski for helpful comments and discussions.
	\newline

	\bibliographystyle{plain}
	\bibliography{MarcinSwieca}
\end{document}